\documentclass[reqno,12pt,hidelinks,oneside]{amsart}
\usepackage[letterpaper]{geometry}
\usepackage{amssymb,amsmath,amsthm,amsxtra,calc,bm, color} 
\usepackage{enumerate}

\usepackage[scr=boondoxo]{mathalfa}
\usepackage{mathtools}
\newcommand{\slfrac}[2]{\left.#1\big/#2\right.}

\usepackage{nccmath} 
\usepackage{cases}

\usepackage{calc}
\usepackage{graphicx}
\usepackage{hyperref}
\usepackage[final]{microtype}

\usepackage{dsfont}

\def\Z{\mathbb{Z}}
\def\Q{\mathbb{Q}}
\def\R{\mathbb{R}}
\def\H{\mathbb{H}}
\def\N{\mathbb{N}}
\def\C{\mathbb{C}}

\newcommand{\abs}[1]{\left|#1\right|}

\newcommand{\slz}{\Gamma(1)}

\newcommand{\abcd}{\ensuremath{\left(\begin{smallmatrix}a&b\\c&d\end{smallmatrix}\right)}}
\newcommand{\ABCD}{\ensuremath{\left(\begin{matrix}a&b\\c&d\end{matrix}\right)}}

\newcommand{\jacsym}[2]{\ensuremath{\left(\frac{#1}{#2}\right)}}
\newcommand{\charup}[2]{\ensuremath{\left(\frac{#1}{#2}\right)}}
\newcommand{\chardown}[2]{\ensuremath{\left(\frac{#2}{#1}\right)}}

\DeclareSymbolFont{rsfscript}{OMS}{rsfs}{m}{n} \DeclareMathSymbol{\Sh}{\mathalpha}{rsfscript}{'123}

\renewcommand{\b}{\mathfrak{b}}
\renewcommand{\c}{\mathfrak{c}}

\newcommand{\isoarrow}{\hookrightarrow\mathrel{\mspace{-15mu}}\rightarrow}

\usepackage[noabbrev,capitalize]{cleveref}

\renewcommand{\hat}{\widehat}
\renewcommand{\tilde}{\widetilde}
\renewcommand{\pmod}[1]{\ensuremath{\hspace{.25em}\left(\operatorname{mod}#1\right)}}

\DeclareMathOperator{\new}{new}

\newtheorem{theorem}{Theorem}[section]
\newtheorem*{theorem*}{Theorem}
\newtheorem{lemma}[theorem]{Lemma}
\newtheorem*{lemma*}{Lemma}
\newtheorem{corollary}[theorem]{Corollary}
\newtheorem{proposition}[theorem]{Proposition}
\newtheorem*{proposition*}{Proposition}

\newtheorem{heuristic}[theorem]{Heuristic}
\theoremstyle{remark}
\newtheorem*{remark}{Remark}

\newtheorem*{definition}{Definition}
\numberwithin{equation}{section}

\title[A Kohnen-Zagier Formula in Level 1]{A Kohnen-Zagier Formula in Level 1}

\date{\today}
\author{Clayton Williams}
\address{Department of Mathematics\\
University of Illinois\\
Urbana, IL 61801} 
\email{cw78@illinois.edu} 

\begin{document}

\begin{abstract}
In 1985 Kohnen used the Shimura correspondence to prove a general formula relating the coefficients of half-integral weight modular forms of level $ 4N $ transforming with respect to the theta multiplier   to central values of $ L $-functions for twists of newforms of level $ N $. This generalized previous work with Zagier. In this paper, we prove a Kohnen-Zagier type formula using Yang's Shimura correspondence for powers of the eta multiplier in level $ 1 $, relating the coefficients of half-integral weight modular forms to central values of $ L $-functions for  twists of newforms of level $ 6 $.  To do so we develop some of the theory of modular forms transforming with respect to powers of the eta multiplier and introduce Shintani lifts of level $ 6 $, with Fourier coefficients given in terms of cycle integrals. 
\end{abstract}

\maketitle

\setcounter{tocdepth}{1}  
\tableofcontents

\section{Introduction}

The central values of modular $ L $-functions are a principal subject of investigation in number theory. A powerful tool for the study of these central values was introduced by Kohnen and Zagier \cite{Kohnen-Zag-1981,Kohnen-1985}, who related the central values of twisted modular $ L $-functions to the coefficients of half-integral weight modular forms. \par

We recall Kohnen and Zagier's results here \cite[Thm. 1]{Kohnen-Zag-1981}, \cite[Cor. 1]{Kohnen-1985}. Let $ \lambda\geq 1 $ and $ N\geq 1 $ be integers with $ N $ odd and square-free. Let  $ f\in S_{2\lambda}^{\operatorname{new}}(\Gamma_0(N)) $ be a newform of weight $ 2\lambda $ and level $ N $.     By \cite[Thm. 2]{Kohnen-iso-1982} there is a Hecke eigenform $ F $  of weight $ \lambda+1/2 $ and level $ 4N $, transforming with respect to a power of the theta multiplier, which corresponds to $ f $ under the Shimura correspondence. This $ F $ has a Fourier expansion given by 
\begin{align} \label{eqn:plus-space-fourier}
F(\tau) = \sum_{\substack{m\geq 1\\(-1)^\lambda m\equiv 0,1\pmod{4}}}a(m)e(m\tau),
\end{align}
where $ e(\tau):=e^{2\pi i \tau} $ and $ \tau\in \mathbb{H} $ (the complex upper half-plane). Cusp forms with such an expansion are said to belong to the plus-space $ S_{\lambda+1/2}^+$. For an integral weight cusp form $ f(z) = \sum_n b(n)e(nz) $, the $ L $-function for $ f $  is given in the region of absolute convergence by the  Dirichlet series
\begin{align*}
L(f,s):=\sum_{n\geq 1}\frac{b(n)}{n^s}. 
\end{align*}

\begin{theorem*}\cite[Cor. 1]{Kohnen-1985}
    Let $ N $ be odd and square-free, and let  $f\in S_{2\lambda}^{\operatorname{new}}(\Gamma_0(N))$ be a newform. Suppose $ F=\sum_m a(m)e(m\tau)\in S_{\lambda+1/2}^+ $ is a Hecke eigenform corresponding to $ f $ under the Shimura correspondence. Let $ \omega(N) $ be the prime-divisor counting function, $ D $ be a fundamental discriminant with $ (-1)^\lambda D=\abs D $, and let $ f\otimes\charup{D}{\cdot} $  denote the twist of $ f $ by $ \charup{D}{\cdot} $. Let $ \langle\cdot,\cdot\rangle $ be the Petersson inner product.  Suppose that, for each prime divisor $ p $ of $ N $, the $ p $th Atkin-Lehner eigenvalue $ \delta_p $ of $ f $ satisfies  $ \charup{D}{p} = \delta_p $.  Then we have
    \begin{align*}
    \frac{\abs{a(\abs{D})}^2}{\langle F,F\rangle} = 2^{\omega(N)}\frac{(\lambda-1)!}{\pi^\lambda}\abs{D}^{\lambda-1/2} \frac{L \left(f\otimes \charup{D}{\cdot},\lambda\right)}{\langle f,f\rangle}.
    \end{align*}

\end{theorem*}

Explicit identities of this type for automorphic forms of various kinds are called Kohnen-Zagier formulae\footnote{These are also called Katok-Sarnak formulae when considering non-holomorphic Maass forms \cite[p. 4]{ILT-Katok-Sarnak-Formula}.}, and have applications to algebraicity and nonvanishing of $ L $-function central values, the theory of elliptic curves, combinatorial generating functions, invariants for real quadratic fields, and quantum unique ergodicity, among others. See, for example, \cite{Ahlgren-Rouse-Critical-L-Modp,Ahlgren-Boylan-CentralVal,Bruinier-1999-Nonvanishing,Liu-Nonvanishin-Maass-2018,Ono-Skinner-19980Halfweight-mod-l,Kane-2010-Ternary,Ono-Soun-1997-Ternary,DukeImamogluToth2016,Lester-Rad-QUE-2020,BO-HeegnerHMF}.
Waldspurger \cite{Walds-1981} also proved a celebrated and general theorem of this type for automorphic representations, but without an explicit constant of proportionality.\par

Shimura's original correspondence \cite{Shimura-originalcorr-1973} and its later refinements (\cite{Niwa-Theta-Kernel}, \cite{Kohnen-iso-1982}, and \cite{Cipra-1983}, among others) dealt with modular forms transforming with respect to the theta multiplier on $ \Gamma_0(4) $; consequently, results leveraging Shimura's correspondence apply for half-integral weight modular forms of levels divisible by $ 4 $. In this paper, we work with modular forms transforming with respect to powers of the eta multiplier on $ \operatorname{SL}_2(\Z) = \Gamma(1) $. Let $ \eta $ be Dedekind's eta function. Define the eta multiplier $\nu_\eta:\slz\to \C^\times $ by 
\begin{align*}
\nu_\eta \left(\ABCD\right):=\frac{\eta \left(\abcd\tau\right)}{(c\tau+d)^{1/2}\eta(\tau)},
\end{align*}
which is a $ 24 $th root of unity independent of $ \tau\in \mathbb H $. A formula for $ \nu_\eta $ can be found in \cite[Thm. 5.8.1]{Cohen-Strom-MF-2017}.  We denote the space of modular forms of weight $ \lambda+1/2 $ and multiplier $ \nu $, vanishing at the cusp, by $ S_{\lambda+1/2}(\slz,\nu) $.

In \cite{Yang2014} Yang showed that there exists a correspondence for modular forms transforming with respect to $ \nu_\eta^r $ when $ (r,6)=1 $ and $ 1\leq r\leq 23 $. 
Let   $ S_{2\lambda}^{\operatorname{new}} (\Gamma_0(6),\delta_2,\delta_3)$ be the space of newforms with Atkin-Lehner eigenvalues $ \delta_2$ and $\delta_3 $, respectively, and suppose $ {S_{\lambda+1/2}\left(\slz,\nu_\eta^r\right) \neq\{0\}}$.
In particular\footnote{While \cite[Thm. 1]{Yang2014} states the result for level $ 576 $, this is after applying the level-raising operator $ V_{24} $. Yang explicitly proves the isomorphism for the spaces mentioned here, as can be seen from \cite[p. 6, Prop. 6, and \S 5]{Yang2014}. 
}, Yang proved the existence of an isomorphism of Hecke modules between $ {S_{\lambda+1/2}\left(\slz,\nu_\eta^r\right)} $ and $ {S_{2\lambda}^{\operatorname{new}}(\Gamma_0(6),-\charup{8}{r},-\charup{12}{r})}$. See \cref{subsec: Modular forms background} for precise definitions.   For brevity we define \begin{gather}
\delta_2(r):=-\charup{8}{r} \hspace{1em}\text{and}\hspace{1em}\delta_3(r):=-\charup{12}{r}. \label{eqn:delta(r) defn}
\end{gather}
Before stating our main theorem we record a consequence of Yang's isomorphism \cite[Cor.~1]{Yang2014}.

\begin{definition}[Shimura Correspondence for Hecke Eigenforms]
Let $ \lambda\geq 2 $, $ (r,6)=1 $, and $ 1\leq r\leq 23 $. Let $ f\in S_{2\lambda}^{\operatorname{new}}(\Gamma_0(6),\delta_2(r),\delta_3(r)) $ be a newform with Hecke eigenvalues $ b(n) $, and for $ (n,6)=1 $ let $ T_{n^2}\in \operatorname{End}(S_{\lambda+1/2}(\slz,\nu_\eta^r)) $ be the $ n^2 $ Hecke operator, defined by \eqref{eqn: Tp2 half int}-\eqref{eqn:Tp2 recur}.  Let $ F\in S_{\lambda+1/2}(\slz,\nu_\eta^r) $ be non-zero. We say that $ F $ corresponds to $ f $ under the Shimura correspondence if  
\begin{align*}
F \Big|T_{n^2} = \charup{12}{n}b(n)F \quad \text{for all } n\geq 1,\,(n,6)=1.
\end{align*}
\end{definition}

In \cref{cor: mult-one-eta-mult} we will prove that such an $ F $ exists, and is unique up to scalar multiplication.
Our main theorem is the following, where $ \langle\cdot,\cdot\rangle $ is the Petersson inner product. Note that, because $ \eta(\tau+1) = e(1/24)\eta(\tau) $, cusp forms $ F\in S_{\lambda+1/2}(\slz,\nu_\eta^r) $ have Fourier expansions of the form
\begin{align}\label{eqn: etamult fourier}
F(\tau) = \sum_{\substack{m\geq 1\\m\equiv r\pmod{24}}}a(m)e \left(\frac{m\tau}{24} \right).
\end{align}

Our main result is a Kohnen-Zagier formula in level $ 1 $.

\begin{theorem}\label{thm:Main-KZ-formula}
 Let $ (r,6)=1 $ with $ 1\leq r\leq 23 $, let $ \lambda\geq 2 $ be an integer, and suppose that $ f\in  S_{2\lambda}^{\operatorname{new}}(\Gamma_0(6),\delta_2(r),\delta_3(r)) $ is a newform, where $ \delta_p(r) $ is given by \eqref{eqn:delta(r) defn}. Suppose also that $ F \in  S_{\lambda+1/2}(\slz,\nu_\eta^r) $ is a Hecke eigenform corresponding to $ f $ under the Shimura correspondence, and that $ F $ has coefficients $ a(m) $ as in \eqref{eqn: etamult fourier}.     Then for all fundamental discriminants $ D $ with $ (-1)^\lambda D=\abs{D}\equiv r\pmod{24} $ we have

\begin{gather}\label{eqn:kz-main}
\frac{\abs{a(\abs{D})}^2}{\left\langle F,F\right\rangle}=\frac{\Gamma(\lambda)\abs{D}^{\lambda-1/2}}{\sqrt{6}(24\pi)^{\lambda}}\frac{L \left(f\otimes \jacsym{D}{\cdot} ,\lambda\right)}{\left\langle f,f\right\rangle}.
\end{gather}

\end{theorem}

Our second main result is a kernel function for the Shimura maps for modular forms transforming with respect to powers of the eta multiplier. Let $ \psi $ be a Dirichlet character modulo $ N $ with $ (N,6)=1 $. For simplicity let $ \lambda\geq 2 $ be an integer and $ (r,6)=1 $, though the following extends to $ \lambda=1 $, $ (r,6)=3 $. Ahlgren, Andersen, and Dicks \cite{AAD-Shimura-2024} introduced Shimura lifts which are Hecke equivariant maps $ \Sh_t $, indexed by positive square-free $ t $, with
\begin{align*}
\Sh_t:S_{\lambda+1/2}(\Gamma_0(N),\psi\nu_\eta^r) \to S_{2\lambda}^{\operatorname{new }2,3}\left(\Gamma_0(6N),\psi^2,\psi(2)\delta_2(r),\psi(3)\delta_3(r)\right).
\end{align*}
Here the target space is new at $ 2 $ and $ 3 $ and $ (N,6)=1 $. 
In fact, a linear combination of the $ \Sh_t $ maps is an isomorphism of Hecke modules when $ N=1 $.
See \cref{subsec: Modular forms background} for details.

For an integral binary quadratic form 
\[Q(x,y)= [a,b,c](x,y) = ax^2+bxy+cy^2\]
 let $\operatorname{disc}(Q):=b^2-4ac$ be the discriminant, which is invariant under the action of $ \slz $.  Denote by $\mathscr{Q}$ the set of all such integral binary quadratic forms, and for $ D\in\Z$, $ N\in\Z_{\geq 1} $   define 
\begin{gather}\label{eqn:fancyQ-defn}
\mathscr{Q}_{D,N}:=\{Q\in\mathscr{Q}:\operatorname{disc}(Q)=D,\,N\mid Q(1,0)\}.
\end{gather}
Let $ D \equiv 1\pmod{4} $ be a fundamental discriminant and $ DD'>0 $ a discriminant with $(DD',6)=1 $. Let $ \omega_D $ be a genus character on $\mathscr Q_{DD',N} $, defined in \eqref{eqn:omega_D}, and suppose $ \lambda\geq 2 $. We define
\begin{gather}\label{eqn:g-quad-form-defn}
g_{\lambda,6}(z;D,D'):=\sum_{\substack{Q=[*,b,*]\in \mathscr{Q}_{DD',6}}}\jacsym{12}{b}\omega_D(Q)Q(z,1)^{-\lambda}.
\end{gather}
This is a modification of \cite[(2)]{Kohnen-1985}. In \cref{subsec:repn kernel quad forms} we will show that $ g_{\lambda,6}\in S_{2\lambda}(\Gamma_0(6)) $.

The kernel function for the Shimura lifts in level $ 1 $  is given by the following theorem.
In what follows let $ \varepsilon_r $ be the Gauss sign, defined for odd $ r $ by 
\begin{align*}
\varepsilon_r:=\begin{cases}
1&r\equiv 1\pmod{4},\\
i&r\equiv -1\pmod{4}.
\end{cases}
\end{align*}

\begin{theorem}\label{thm: kernel-function}
     Let $ (r,6)=1 $ with $ 1\leq r\leq 23 $, let $ \lambda\geq 2 $ be an integer, and let $ D $ be a fundamental discriminant with $ (-1)^\lambda D=\abs{D}\equiv r\pmod{24} $. Let $ \beta $ be the Euler beta function, and write
    \begin{gather}\label{eqn:kappa}
        \kappa_{\lambda,r,D}:=-i^{-\lambda}\varepsilon_{r}\jacsym{12}{r}\frac{\beta(\lambda,-1/2)}{4\sqrt{2}\pi}\left(\abs{D}/12 \right)^{\lambda-1/2}.
    \end{gather}
    For $z,\tau\in \mathbb{H}$ define
    \begin{align}\label{eqn: Kernel-quad-repn}
    \Omega_{\lambda,r,D}(z;\tau)=\Omega_{D}(z;\tau)
    :=\kappa_{\lambda,r,D}
    \sum_{\substack{m\geq 1\\m\equiv r\pmod{24}}}m^{\lambda-1/2}
   g_{\lambda,6}(z;D,(-1)^{\lambda}m)
    e \left(\frac{m\tau}{24} \right).
    \end{align}
    Suppose $ S_{\lambda+1/2}(\slz,\nu_\eta^r)\neq \{0\} $. 
    Then for fixed $ \tau\in\mathbb H $ we have $\Omega_{\lambda,r,D}(z;\tau)\in S_{2\lambda}(\Gamma_0(6))  $ as a function of $ z\in\mathbb{H} $, and for fixed $ z \in \mathbb{H}$ we have $ \Omega_{\lambda,r,D}(z;\tau)\in S_{\lambda+1/2}(\slz,\nu_\eta^r) $ as a function of $ \tau\in\mathbb{H} $. Moreover, for $F(\tau)\in S_{\lambda+1/2}(\slz,\nu_\eta^r)$ we have
    \begin{gather*}
    \left\langle F(\cdot),\Omega_{\lambda,r,D}(-\overline{z};\cdot)\right\rangle = \Sh_{\abs{D}}(F)(z).
    \end{gather*}

\end{theorem}  

This also gives us a way to compute the maps which are adjoint to the $ \Sh_{\abs{D}} $, called      Shintani lifts after \cite{Shintani-1975-cusp-forms}.

\begin{corollary}\label{cor: shintani-property}
Let $ \lambda\geq 2$  and $ (r,6)=1 $ with $ 1\leq r\leq 23 $. Let $ D $ be a fundamental discriminant with $ (-1)^{\lambda}D=\abs D\equiv r\pmod{24} $.   Suppose $ S_{\lambda+1/2}(\slz,\nu_\eta^r)\neq \{0\} $ and let $ {\Sh_{\abs{D}}^*:S_{2\lambda}^{\operatorname{new}}(\Gamma_0(6))\to S_{\lambda+1/2}(\slz,\nu_\eta^r)} $ be the $ D $th Shintani lift, defined  by
\begin{align}\label{eqn:shintani-defn}
    \Sh^*_{\abs{D}}(f)(\tau):=\left\langle f(\cdot),\Omega_{\lambda,r,D}(\cdot;-\overline{\tau})\right\rangle.
\end{align}
Then for all $ f\in S_{2\lambda}^{\new}(\Gamma_0(6),\delta_2(r),\delta_3(r)) $ and all $ F\in S_{\lambda+1/2}(\slz,\nu_\eta^r) $ we have
\begin{align}\label{eqn:shim-shin-adjoint}
\langle \Sh_{\abs{D}}(F),f\rangle = \langle F,\Sh_{\abs{D}}^*(f)\rangle.
\end{align}
  
\end{corollary}

Our last main result is a representation of the Shintani lifts in terms of cycle integrals of modular forms.
For $Q\in \mathscr{Q}$ with positive discriminant let $S_Q$ be the geodesic in $\mathbb{H}$ connecting the two roots of $Q(z,1) $ in $\overline{\mathbb{H}}:= \mathbb{H}\cup\mathbb R\cup\{i\infty\} $ with orientation given in \cref{sec: cycle-integrals}.  Let $C_Q$ be the image of $S_Q$ in ${\Gamma_0(6)}\backslash{\overline{\mathbb{H}}}$. For $ \lambda\geq 2 $ we define a differential $ d_{Q,\lambda} z$ in \eqref{eqn:dQdiff}.
For $f\in S_{2\lambda}(\Gamma_0(6))$, $D$ a fundamental discriminant with $(-1)^\lambda D>0$, and $m>0$ with $ \abs{D} m $ a discriminant coprime to $ 6 $,  define
\begin{gather}\label{eqn:s-cycle-sum}
s_{\lambda,6}(f;D,(-1)^\lambda m):=\sum_{\substack{Q=[*,b,*]\in \mathscr{Q}_{\abs{D}m,6}\pmod{\Gamma_0(6)}}}\jacsym{12}{b}\omega_D(Q)\int_{C_Q}f(z)d_{Q,\lambda}z,
\end{gather}
where the sum is taken over an inequivalent set of equivalence class representatives for the action of $\Gamma_0(6)$ on $\mathscr{Q}$.  A crucial step in the proof of \cref{thm:Main-KZ-formula} is relating the coefficients of the half-integral weight form to $s_{\lambda,6}$, which we then relate to central values of $L$-functions. Compare \cite[Thm. 3]{Kohnen-1985}. The sign depends on whether or not the discriminant $ \abs{D}m $ is a perfect square, so let $ \mathds{1}_\square(m)=1 $ if $ \sqrt{\abs{m}}\in\Z$ and $ \mathds{1}_\square(m)=0 $ otherwise.

\begin{theorem}\label{thm: cycle-int-and-coeff}
Let $ (r,6)=1 $ with $ 1\leq r\leq 23 $, let $ \lambda\geq 2 $ be an integer, and let $ D $ be a fundamental discriminant with $ (-1)^\lambda D=\abs{D}\equiv r\pmod{24} $. Let $ f\in S_{2\lambda}^{\operatorname{new}}(\Gamma_0(6),\delta_2(r),\delta_3(r)) $ be a newform and suppose $ F \in S_{\lambda+1/2}(\slz,\nu_\eta^r)$ is a Hecke eigenform corresponding to $ f $ under the Shimura correspondence. Let $ F $ have Fourier coefficients $ a(m) $  as in \eqref{eqn: etamult fourier}.  Then  for $ m\geq 1 $ with $ m\equiv r\pmod{24} $ we have the identity       
\begin{align}\label{eqn: coeff and cycle sum}
\frac{a(m)\overline{a(\abs D)}}{\left\langle F,F\right\rangle} =\frac{(-1)^{\mathds{1}_\square(\abs{D}m)}i^{-\lambda}\varepsilon_r}{4\sqrt{6}(12)^{\lambda}}\jacsym{12}{r}  \frac{s_{\lambda,6}(f,D,(-1)^\lambda m)}{\langle f,f\rangle}.
\end{align}

\end{theorem}

\subsection{Outline of the Paper} In \cref{sec:numerical-demo} we provide numerical illustrations of \cref{thm:Main-KZ-formula}, where the quantities $ \langle f,f\rangle,\,\langle F,F\rangle,$ and $ L \left(f\otimes \charup{D}{\cdot},\lambda\right) $ are computed or approximated  independently. 
In \cref{subsec: Modular forms background} we give a brief review of the theory of modular forms transforming with respect to a power of $ \nu_\eta $, as well as necessary background from \cite{Yang2014} and \cite{AAD-Shimura-2024}.  In \cref{subsec:background-quad-form} we review the theory of integral binary quadratic forms and introduce the genus character $ \omega_D $.\par 

In \cref{sec:Poincare} we develop properties of Poincar\'e series on $ S_{\lambda+1/2}(\slz,\nu_\eta^r) $; in particular, we give an application of unfolding to compute the Fourier coefficients of half-integral weight modular forms. This lets us derive a representation for $ \Sh_{\abs{D}} $ through integration against Poincar\'e series in \cref{prop:Kernel-Poincare}. In \cref{subsec:repn kernel quad forms} we compute the Fourier series for $ g_{\lambda,6} $, which lets us prove \cref{thm: kernel-function} and \cref{cor: shintani-property} in \cref{subsec:proof-equiv-kernel}. These proofs rely on an identity, \cref{thm: relate-Kloost-Weyl}, relating Kloosterman sums and certain twisted quadratic Weyl sums. The proof of \cref{thm: relate-Kloost-Weyl} is quite substantial and uses the theory of the Weil representation attached to a finite quadratic module, and so is proved in separate, concurrent work of the author \cite{Williams2026Powers}.\par

In \cref{sec: cycle-integrals} we review the theory of cycle integrals and prove a formula, \cref{prop: cycle-int-shintani}, for the Fourier coefficients of $ \Sh_{\abs{D}}^*(f) $ in terms of sums of these cycle integrals. This, together with the multiplicity one property, enables us to prove \cref{thm: cycle-int-and-coeff}. Finally, we relate sums of cycle integrals to central values of twisted $ L $-functions, which ultimately leads to a proof of \cref{thm:Main-KZ-formula}.

In \cref{sec:appx-numerical-central} we describe an algorithm for numerically approximating central values, and in \cref{sec:appx-numerical-petersson} we give a heuristic which is useful for approximating Petersson inner products on $ S_{\lambda+1/2}(\slz,\nu_\eta^r) $.

\begingroup
\renewcommand{\addcontentsline}[3]{}
\section*{Acknowledgments}
\endgroup
The author thanks Scott Ahlgren for many helpful discussions in the preparation of this paper. 

\subsubsection*{AI Usage} The mathematics and writing of this paper were completed without AI tools. 
ChatGPT was used to proofread the paper.
Its most significant contribution was catching an inconsistency related to the orientation convention for the cycle integrals in \cref{sec: cycle-integrals}.

\section{Numerical Illustrations}\label{sec:numerical-demo}
Here we illustrate \cref{thm:Main-KZ-formula} when $ r=5,7,11 $  using the numerical methods of the appendix, which are summarized as follows:
\begin{itemize}
\item The inner products $ \langle f,f\rangle $ for newforms of level $ 6 $ are computed using modular symbols in Pari \cite{PARI2}, \cite[\S 12]{Cohen2018pariMF}.
\item The inner products $ \langle F,F\rangle $ for half-integral weight $ F $ are approximated in Mathematica \cite{Mathematica} using a heuristic derived from the Rankin-Selberg method and Perron inversion with $ 100\,000 $ coefficients. See  \cref{prop: inner product numerical} and \cite{lowryduda2024numerical}.
\item The central values of the $ L $-functions $ L_D:=L \left(f\otimes \charup{D}{\cdot},\lambda\right) $   are approximated in Mathematica using an integral representation for the completed $ L $-function with $ 40\,000 $ coefficients (\cref{sec:appx-numerical-central}). The coefficients of the newforms $ f $ are computed using the Sage commands available on LMFDB \cite[\href{https://www.lmfdb.org/ModularForm/GL2/Q/holomorphic/6/4/a/a/}{Newform Orbits 6.4.a.a},\href{https://www.lmfdb.org/ModularForm/GL2/Q/holomorphic/6/6/a/a/}{ 6.6.a.a},\href{https://www.lmfdb.org/ModularForm/GL2/Q/holomorphic/6/10/a/a/}{ 6.10.a.a}]{lmfdb}, \cite{sagemath}.
\end{itemize}

In \cref{fig:numerical-demo}  (with data rounded for clarity) we compare the Fourier coefficients $ a(|D|) $ of $ F\in S_{\lambda+1/2}(\slz,\nu_\eta^r) $ to
\begin{align*}
\operatorname{RHS}(D) := \frac{\Gamma(\lambda)\abs{D}^{\lambda-1/2}}{\sqrt{6}(24\pi)^{\lambda}}\frac{\langle F,F\rangle L \left(f\otimes \jacsym{D}{\cdot} ,\lambda\right)}{\left\langle f,f\right\rangle},
\end{align*}
  where $ f \in S_{2\lambda}^{\operatorname{new}}(\Gamma_0(6),\delta_2(r),\delta_3(r)) $ is the newform corresponding to $ F $ under the Shimura correspondence. Each $ f $ is associated with the LMFDB label for its newform orbit (each orbit is one dimensional) \cite{lmfdb}, and in each case the Atkin-Lehner signs satisfy $ \jacsym{D}{p}\delta_p(r) = -1 $ for $ p=2,3 $.
\begin{figure}[htbp]
    \centering 
    \begin{tabular}{@{}c@{\hspace{8pt}}c@{\hspace{8pt}}c@{\hspace{8pt}}c@{\hspace{8pt}}c@{}}
        \begin{minipage}[t]{0.30\textwidth}
        \centering
        \begin{gather*}
            F=\eta^5\\
            f=\operatorname{LMFDB}\text{ }6.4.a.a\\
            \langle F,F\rangle \approx 0.035689032784253\\
            \langle f,f\rangle \approx0.00009487663999
        \end{gather*} 
        \resizebox{1.05 \linewidth}{!}{
        \begin{tabular}{|c||c|c||c|}
            \hline
            \(D\) & \(a(\abs{D})^2\) &\(\operatorname{RHS}(D)\) & $L_D$\\
            \hline\hline
            5   & 1     & 0.99137& 3.2825\\
            \hline
            29  & 25    & 24.784 &5.8749 \\
            \hline
            53  & 25   & 24.784  & 2.3778 \\
            \hline
            77  & 100     & 99.137&5.4315  \\
            \hline
        \end{tabular}
        }
  \vspace{.25em}  \[ a(77)^2/\operatorname{RHS}(77) = 1.0087. \]
    \end{minipage}
        &
        \vrule width 0.8pt
        &
       \begin{minipage}[t]{0.30\textwidth}
        \centering
        \begin{gather*}
            F=\eta^7\\
            f  = \operatorname{LMFDB}\text{ }6.6.a.a\\
            \langle F,F\rangle  \approx 0.011332018662970\\
            \langle f,f\rangle \approx 0.00001531540537
        \end{gather*}
         \resizebox{1.05 \linewidth}{!}{
        \begin{tabular}{|c||c|c||c|}
            \hline
            \(D\) & \(a(\abs{D})^2\) &\(\operatorname{RHS}(D)\) & $L_D$\\
            \hline\hline
            -7   & 1     & 0.99117& 5.4244\\
            \hline
            -31  & 49    & 48.567 &6.4400\\
            \hline
            -55  & 196   & 194.27 & 6.1439 \\
            \hline
            -79  & 49     & 48.567&0.62119  \\
            \hline
        \end{tabular}
        }
  \[ a(79)^2/\operatorname{RHS}(-79) = 1.0089. \]
    \end{minipage}
     &
        \vrule width 0.8pt
        &
       \begin{minipage}[t]{0.30\textwidth}
        \centering
        \begin{gather*}
            F=\eta^{11}\\
            f = \operatorname{LMFDB}\text{ }6.10.a.a  \\
            \langle F,F\rangle  \approx 0.001231947342375\\
            \langle f,f\rangle \approx 0.00000242229
        \end{gather*}
      \resizebox{1.05 \linewidth}{!}{
        \begin{tabular}{|c||c|c||c|}
            \hline
            \(D\) & \(a(\abs{D})^2\) &\(\operatorname{RHS}(D)\) & $L_D$\\
            \hline\hline
            -11   & 1     & 1.0099& 10.1694\\
            \hline
            -35  & 121   & 122.19 &6.73044\\
            \hline
            -59  & 1936   & 1955.1  & 10.2715 \\
            \hline
            -83  & 3025     &3054.8&3.45492  \\
            \hline
        \end{tabular}
        }
 \[ a(83)^2/\operatorname{RHS}(-83) = 0.99024. \]
    \end{minipage}
    \end{tabular}
    \caption{Demonstrating \cref{thm:Main-KZ-formula} when the newspace is one dimensional}
    \label{fig:numerical-demo}
\end{figure}

\subsection{A Newform Orbit of Dimension 2}
We give a more interesting example for a space with $ \operatorname{dim}(S_{2\lambda}^{\operatorname{new}}\left(\Gamma_0(6),\delta_2,\delta_3\right))>1 $.
 
We have $ S_{25/2}(\slz,\nu_\eta) = \operatorname{span}\{E_{12}\eta,\Delta\eta\} $, where $ E_{12} = 1+ \frac{65520}{691} q+...$ is the normalized Eisenstein series of weight $ 12 $  and $ \Delta = \eta^{24} $ is the discriminant form. To compute eigenforms we find the two one dimensional eigenspaces of $ T_{25} $, which suffices because the $ T_{p^2} $ commute for $ p\geq 5 $.    Using \eqref{eqn: Tp2 half int} we find that the transformation matrix for $ T_{25} $ acting on the basis $ \{E_{12}\eta,\Delta\eta\} $ is
\begin{align*}
M_{25}:=\frac{1}{691}\left(\begin{matrix}
 -33740169546 & 691 \\
 \frac{16012327034465169408000}{691} & 16293693750
\end{matrix}
\right).
\end{align*}
 
Computing the eigenvectors for $ M_{25} $ we find the eigenforms
\begin{align*}
F_1(\tau)=\sum_{\substack{m\geq 1\\m\equiv 1\pmod{24}}}a_1(m)e \left(\frac{m\tau}{24} \right):= E_{12}\eta+\frac{5184}{691} \left(4825797+691 \sqrt{1296640489}\right)\Delta\eta,\\
F_2(\tau)=\sum_{\substack{m\geq 1\\m\equiv 1\pmod{24}}}a_2(m)e \left(\frac{m\tau}{24} \right):=E_{12}\eta+\frac{5184}{691} \left(4825797-691 \sqrt{1296640489} \right)\Delta\eta.
\end{align*} 
$ S_{24}^{\operatorname{new}} (\Gamma_0(6),-1,-1) $ is the span of the two real embeddings \cite[\href{https://www.lmfdb.org/ModularForm/GL2/Q/holomorphic/6/24/a/d/}{Newform Orbit 6.24.a.d}]{lmfdb} of 
\begin{multline*}
f_{\beta}(z):=e(z)+2048 \,e(2z)+177147 \,e(3z)+4194304 \,e(4z)
+(\beta -179195) \, e(5z)+\dots,
\end{multline*}
where $ \beta $ satisfies
\begin{align*}
\beta ^2-25606546 \beta -34681805985643055 = 0.
\end{align*} 
 The coefficient field of $ f_\beta $ is $ \Q(\beta) $. Write $ f_1 $ for the newform \cite[\href{https://www.lmfdb.org/ModularForm/GL2/Q/holomorphic/6/24/a/d/}{6.24.a.d.1.1}]{lmfdb}, which corresponds to $ \beta=12803273 - 5184 \sqrt{1296640489} $, and $ f_2 $ for the conjugate newform \cite[\href{https://www.lmfdb.org/ModularForm/GL2/Q/holomorphic/6/24/a/d/}{6.24.a.d.1.2}]{lmfdb}  under the nontrivial automorphism in $ \operatorname{Gal}(\Q(\beta)) $.\par  

The images of $ F_1 $ and $ F_2 $ under $ \Sh_{D} $ for any square-free discriminant $ D>0 $ with ${ D\equiv 1\pmod{24} }$ can be computed coefficient-wise with \eqref{eqn:AAD-Shim} below. In particular, $ F_j $ corresponds to $ f_j $  through the Shimura correspondence for $ j=1,2 $. By \eqref{eqn: St scalar mult} we have $ \Sh_{D}(F_j) = a_j(D)f_j $ for $ j=1,2 $. Computing the lifts $ \Sh_{73}(F_j) $ we find that
\begin{align*}
\Sh_{73}(F_1)(z) = a_1(73)f_{1}(z),\\
   \Sh_{73}(F_2)(z) = a_2(73)f_{2}(z).
\end{align*}

We demonstrate \cref{thm:Main-KZ-formula} with the following numerical data, rounded for clarity:
\begin{itemize}
\item The inner products
\begin{align*}
\langle f_{1},f_1\rangle \approx 0.0000855682823,\\
\langle f_2,f_2\rangle \approx 0.0000751634555,
\end{align*}
are computed using modular symbols in Pari.
\item The inner products
\begin{align*}
\langle F_1,F_1\rangle  \approx 5.11382\times 10^{10},\\
\langle F_2,F_2\rangle \approx 3.47386\times 10^{10},
\end{align*}
are approximated with $ 100\,000 $ coefficients in Mathematica, using the heuristic of \cref{sec:appx-numerical-petersson}.
\item The central values
\begin{align*}
L \left(f_1\otimes \charup{73}{\cdot},12\right) \approx 4.939297819,\\
L \left(f_2\otimes\charup{73}{\cdot},12\right)\approx 2.925193375,
\end{align*}
 are computed with $ 45\,000 $ coefficients in Mathematica, using the methods of \cref{sec:appx-numerical-central}.

\end{itemize}
 With these data, we have
\begin{gather*}
a_1(73)^2\Big/\left(\frac{\Gamma(12)(73)^{12-1/2}}{\sqrt{6}(24\pi)^{12}}\frac{\langle F_1,F_1\rangle L \left(f_1\otimes \charup{73}{\cdot},12\right)}{\langle f_1,f_1\rangle} \right)\approx 0.983925,\\
a_2(73)^2\Big/\left(\frac{\Gamma(12)(73)^{12-1/2}}{\sqrt{6}(24\pi)^{12}}\frac{\langle F_2,F_2\rangle L \left(f_2\otimes \charup{73}{\cdot},12\right)}{\langle f_2,f_2\rangle} \right)\approx 0.977854.
\end{gather*}

\section{Background on Modular and Quadratic Forms}

\subsection{Modular Forms for the Eta Multiplier}\label{subsec: Modular forms background}
If $ \kappa\in\frac{1}{2}\Z $ and $ \gamma = \abcd\in\operatorname{GL}_2^+(\R) $, define the slash operator by
\begin{align*}
\left(f \Big|_{\kappa}\gamma\right)(z) = \left(\det\gamma\right)^{\tfrac{\kappa}{2}}(cz+d)^{-\kappa}f(\gamma z). 
\end{align*}
We say $ f $ is modular of weight $ \kappa $ for a subgroup $ \Gamma $ of $ \slz $ with multiplier $ \nu $ if     
\begin{align*}
\left(f \Big|_{\kappa}\gamma\right) (z) = \nu(\gamma)f(z)
\end{align*}
for all $ \gamma \in\Gamma$, $ z\in\mathbb H $.
\par  

 Let $ \Gamma\leq \slz $ be a subgroup of finite index and $ \kappa\in\frac{1}{2}\Z $. Let $ z=x+iy $ be decomposed into real and imaginary parts and $ \mathfrak F_\Gamma $ be a simply connected fundamental domain for the action of $ \Gamma $ on $ \H $. Then the Petersson inner product on $ S_\kappa(\Gamma,\nu) $ is
\begin{align*}
\langle f,g\rangle:=\frac{1}{[\slz:\Gamma]}\int_{\mathfrak F_\Gamma}f(z)\overline{g(z)}y^\kappa d\mu, 
\end{align*}
where $ d\mu:=\frac{dxdy}{y^2}  $ is the hyperbolic measure on $ \mathbb H $. When $ \Gamma=\slz $ we will write $ \mathfrak{F}_\Gamma = \mathfrak F $. \par

\begin{remark}Note that if $ (r,2)=1 $ we have $ S_{\lambda+1/2}\left(\slz,\nu_\eta^r\right)=\eta^r M_{\lambda+1/2-r/2}(\slz) $, and $ M_{\lambda+1/2-r/2}(\slz)  $ is nonzero exactly when $ \lambda+1/2-r/2\in\{0,4,6,8,...\} $.  This implies the consistency condition   
\begin{align}\label{eqn:consistency}
\lambda+1/2-r/2\equiv 0\pmod{2}.
\end{align}

\end{remark}

We now define the Hecke operators on $ S_{2\lambda}(\Gamma_0(N)) $. Let $ p $ be prime with $ (p,N) =1$; then the $ p $th Hecke operator $ T_p $ is defined by  
\begin{align}\label{eqn:Tp int} 
\left(\sum_{n\geq 0}b(n)e(nz)\right)\Bigg|T_p := \sum_{n\geq 0}\left(b(pn)+p^{2\lambda-1}b(n/p)\right)e(nz),
\end{align}
where $ b(n/p)=0 $ if $ n/p\not\in\Z $ \cite[Cor. 10.2.6]{Cohen-Strom-MF-2017}. Let $ m,n\in\Z_{\geq 1} $ with $ T_1 $ the trivial identity operator and $  \ell\geq 1$.     We define the operators for non-prime indices by the following relations \cite[Thm. 10.2.9]{Cohen-Strom-MF-2017}
\begin{gather}
T_m T_n :=T_{mn} \text{ if }(m,n)=1,\\
T_{p^{\ell+1}}:=T_p T_{p^\ell}-p^{2\lambda-1}T_{p^{\ell-1}}.\label{eqn:Tprecur}
\end{gather}

For $ S_{\lambda+1/2}(\slz, \nu_\eta^r) $ and $ p\geq 5$ a prime   define $ T_{p^2} $ by \cite[(3.16)]{AAD-Shimura-2024}
\begin{multline}\label{eqn: Tp2 half int}
\left(\sum_{\substack{m\geq r\\m\equiv r\pmod{24}}}a(m)e \left(\frac{m\tau}{24} \right)\right)\Bigg|T_{p^2}
 \\:= \sum_{\substack{m\geq r\\m\equiv r\pmod{24}}}\left(a(p^2 m)+\charup{-1}{p}^{\frac{r-1}{2}}\charup{12m}{p}p^{\lambda-1}a(m)
 +p^{2\lambda-1}a \left(\frac{m}{p^2} \right)\right)e \left(\frac{m\tau}{24} \right).
\end{multline} 
Operators of general index coprime to $ 6 $  are defined by the relations
\begin{gather}
T_{m^2}T_{n^2} := T_{m^2n^2}\text{ if }(m,n)=1,\\
T_{p^{2\ell+2}}:= T_{p^2}T_{p^{2\ell}}-p^{2\lambda-1}T_{p^{2\ell-2}}.\label{eqn:Tp2 recur}
\end{gather}
We say that a cusp form in $ S_{2\lambda}(\Gamma_0(N)) $  is a Hecke eigenform if it is a simultaneous eigenfunction for $ \{T_n:(n,N)=1\} $. A cusp form in $ S_{\lambda+1/2}(\slz,\nu_\eta^r) $ is a Hecke eigenform if it is a simultaneous eigenfunction for $ \{T_{n^2}:(n,6)=1\} $.

Fix $ N>0 $ and let $ Q\mid\mid N $ be an exact divisor of $ N $. Let $ W_Q $ be a matrix satisfying 
\begin{align*}
W_Q = \left(\begin{matrix}
Q\alpha&\delta\\N\gamma&Q\beta
\end{matrix}
\right),\quad\det W_Q = Q,\quad\alpha,\delta,\gamma,\beta\in\Z. 
\end{align*}
For such a matrix $ W_Q $ define 
$ \big|_\kappa W_Q $ to be the $ Q $th Atkin-Lehner operator on $ S_\kappa(\Gamma_0(N)) $.  When $ Q = N $ a valid choice is $ W_N = \left(\begin{smallmatrix}
0&-1\\N&0
\end{smallmatrix}
\right) $, and the operator $ \big|_\kappa W_N $ is the Fricke involution.   \par 

If $ d\mid N $ and $ f\big|V_d(z):= f(dz)$ is the level-raising operator, then the old subspace of $ S_{2\lambda}(\Gamma_0(N)) $ is given by \cite[Defn. 13.2.9]{Cohen-Strom-MF-2017}
\begin{align*}
    S_{2\lambda}^{\operatorname{old}}(\Gamma_0(N)):=\sum_{\substack{M\mid N\\M<N}}\sum_{d\mid  N/M}S_{2\lambda}\left(\Gamma_0(M)\right)\Big|V_d.
\end{align*}   
The newspace $ S_{2\lambda}^{\operatorname{new}}(\Gamma_0(N)) $  is the orthogonal complement to $ S_{2\lambda}^{\operatorname{old}}(\Gamma_0(N)) $ under the Petersson inner product. A non-zero Hecke eigenform in the newspace, normalized so the coefficient of $ e(z) $ is $ 1 $, is called a newform (cf. \cite[Defn. 13.3.4 \& Prop. 13.3.11]{Cohen-Strom-MF-2017}). Such a modular form is also an eigenform of $\big|_{2\lambda}W_N $ and $\big |_{2\lambda}W_Q $ for $ Q\mid\mid N $. For a newform $ f\in S^{\operatorname{new}}_{2\lambda}(\Gamma_0(N))$ we denote the eigenvalue of $ f $ under $ \big|_{2\lambda}W_Q $ by $ \delta_Q=\pm 1 $.\par

We have the following theorem, which is a specialization to level 1 of the main result of \cite{AAD-Shimura-2024}, which also treats the cases $ (r,6)=3 $, $ \lambda =1 $, and general level with nebentypus.

\begin{theorem*}\cite[Thm. 1.1]{AAD-Shimura-2024}.
    Let $ 1\leq r\leq 23 $ with $ (r,6)=1 $ and $ t>0 $ be a square-free integer. Suppose that $ \lambda\geq2 $ is an integer and $ F(\tau) = \sum_{m\equiv r\pmod{24}}a(m)e \left(\tfrac{m\tau}{24}\right)\in S_{\lambda+1/2}(\slz,\nu_\eta^r). $   Define $ b(n) $ by
    \begin{align}\label{eqn:AAD-Shim}
    \sum_{n\geq 1}\frac{b(n)}{n^s} = L\left(s-\lambda+1,\chardown{t}{\cdot}\right)\sum_{n\geq 1}\left(\frac{12}{n} \right)\frac{a(tn^2)}{n^s}.  
    \end{align}
    Then
    \begin{align*}
    \Sh_t(F)(z):=\sum_{n\geq 1}b(n)e(nz)\in S_{2\lambda}^{\operatorname{new}}\left(\Gamma_0(6),\delta_2(r),\delta_3(r)\right).
    \end{align*} 
The map is Hecke equivariant; that is, for primes $ p\geq 5 $ we have
\begin{align}\label{eqn:St-hecke}
\Sh_t\left(F \Big|T_{p^2}\right) = \charup{12}{p}\Sh_t(F)\Big|T_p.
\end{align}

\end{theorem*}

The following is a specialization of Yang's isomorphism of Hecke modules in level $ 1 $ \cite[\S 5]{Yang2014}.

\begin{theorem}
Let $ (r,6)=1 $ with $ 1\leq r\leq 23 $ and let $ \lambda\geq 1 $ be an integer.  Suppose $ \lambda+1/2-r/2$ is an even non-negative integer.     For all $ n\geq 1$ with $ (n,6)=1 $ we have
\begin{align*}
\operatorname{tr}\left(S_{\lambda+1/2}(\slz,\nu_\eta^r)\Big| T_{n^2}\right) = \charup{12}{n}\operatorname{tr}\left(S_{2\lambda}^{\operatorname{new}}\left(\Gamma_0(6),\delta_2(r),\delta_3(r)\right)\Big| T_n\right).
\end{align*}

\end{theorem}
Because the modules in the above theorem are semisimple \cite[p. 6]{Yang2014}, a classical theorem  \cite[XVII, Cor. 3.8]{Lang-Algebra} then implies an isomorphism of Hecke modules, which is equivalent to \cite[Thm. 1]{Yang2014} through Proposition 6 (loc. cit.).

\begin{corollary}\label{cor:iso-Hecke}
    Let $ (r,6)=1 $ with $ 1\leq r\leq 23 $ and let $ \lambda\geq 1 $ be an integer.   Suppose $ S_{\lambda+1/2}(\slz,\nu_\eta^r)\neq \{0\} $. There exists a vector space isomorphism
\begin{align*}
\varrho:S_{\lambda+1/2}(\slz,\nu_\eta^r)\isoarrow S_{2\lambda}^{\operatorname{new}}\left(\Gamma_0(6),\delta_2(r),\delta_3(r)\right) 
\end{align*}
such that for all $ F\in S_{\lambda+1/2}(\slz,\nu_\eta^r) $ we have
\begin{align*}
\varrho \left(F \Big|T_{n^2}\right) = \charup{12}{n}\varrho \left(F\right)\Big|T_n \quad\text{for all }n\geq 1,\,(n,6)=1.
\end{align*}

\end{corollary}

\begin{remark}

Recall that multiplicity one for newforms guarantees that simultaneous eigenspaces for the Hecke operators are all one dimensional, so we have the following decomposition into distinct Hecke eigenspaces:
\begin{align*}
S_{2\lambda}^{\operatorname{new}}(\Gamma_0(6),\delta_2(r),\delta_3(r)) = \bigoplus_{\substack{f\text{ a newform }\\ f\in S_{2\lambda}^{\operatorname{new}}(\Gamma_0(6),\delta_2(r),\delta_3(r))}}\operatorname{span}_{\C}(\{f\}).
\end{align*}
\end{remark}

Here we write an explicit version of multiplicity one for $ S_{\lambda+1/2}(\slz,\nu_\eta^r) $,  together with an essential property of the Shimura lifts. The following results all extend to $ \lambda\geq 1 $ and $ (r,6)=3 $ with some modification.  

\begin{corollary}\label{cor: mult-one-eta-mult} Let $ (r,6)=1 $ with $ 1\leq r\leq 23 $ and let $ \lambda\geq 2 $ be an integer.    Suppose $ S_{\lambda+1/2}(\slz,\nu_\eta^r)\neq \{0\} $.  Then the following are true:
\begin{enumerate}
\item Let $ f\in S_{2\lambda}^{\operatorname{new}}(\Gamma_0(6),\delta_2(r),\delta_3(r))$ be a newform with Hecke eigenvalues $ {\{b(n):(n,6)=1\}} $. Then there exists a Hecke eigenform  $ F\in S_{\lambda+1/2}(\slz,\nu_\eta^r) $, unique up to multiplication by a scalar, with
\begin{align}\label{eqn:mult one in lambda plus 1/2}
F \Big|T_{n^2} = \charup{12}{n}b(n)F\quad \text{for all }n\geq 1,\,(n,6)=1.
\end{align} 

    \item Suppose $ F,G\in S_{\lambda+1/2}(\slz,\nu_\eta^r)$ are Hecke eigenforms with eigenvalues under $\big| T_{p^2} $  which are equal for all but a finite number of primes $ p\geq 5 $.  Then $ G=\beta F $ for some $ \beta\in\C $.

    \item   If $ f\in S_{2\lambda}^{\operatorname{new}}(\Gamma_0(6),\delta_2(r),\delta_3(r))$ is a newform corresponding to $F\in S_{\lambda+1/2}(\slz,\nu_\eta^r)$ as above and $ t>0 $ is square-free then    
\begin{align}\label{eqn: St scalar mult}
\Sh_{t}(F) = a(t)f.
\end{align}

    \end{enumerate}

\end{corollary}

\begin{proof}
(1) follows from the Hecke module isomorphism in \cref{cor:iso-Hecke}. (2) follows from \cref{cor:iso-Hecke} together with multiplicity one for newforms. (3) follows from Hecke equivariance of $ \Sh_t $.
\end{proof}

\begin{remark} 
    Let $ F=\sum_m a(m)e \left(\frac{m\tau}{24} \right)\in S_{\lambda+1/2}(\slz,\nu_\eta^r) $ be non-zero. Then $ F \Big|V_{24} $ is a modular form of weight $ \lambda+1/2 $ transforming with respect to a power of the theta multiplier, with a character $ \chi $. By 
Vign\'eras \cite[Thm. 3]{Vigneras-Gamma-Fonctionnelles}, we have that, if $ F \Big|V_{24} $ is not a linear combination of theta functions, then there exist infinitely many square-free integers $ t $ such that $ a(td^2)\neq 0 $ for some $ d\in\Z $. When $ \lambda\geq 2 $ we see that $ F\big| V_{24} $ is not a linear combination of theta series. Using \eqref{eqn: Tp2 half int}-\eqref{eqn:Tp2 recur} we find that if $ F $ is a Hecke eigenform then there are infinitely many $ t>0 $ which are square-free such that $ a(t)\neq0 $. The argument of \cite[p. 260]{Kohnen-iso-gamma4-1980}, together with the identity \eqref{eqn: St scalar mult},  then shows that there exists a linear combination of the lifts $ \Sh_t:S_{\lambda+1/2}(\slz,\nu_\eta^r)\to S_{2\lambda}^{\operatorname{new}}(\Gamma_0(6),\delta_2(r),\delta_3(r))$ which is a Hecke module isomorphism.
\end{remark}

\subsection{Quadratic Forms}\label{subsec:background-quad-form}
Recall the definition of $ \mathscr{Q} $ from \eqref{eqn:fancyQ-defn}. There is a right action of $ \operatorname{GL}_2^+(\Z) $  on $ \mathscr{Q} $ which, for $ Q\in \mathscr{Q} $ and $\left(\begin{smallmatrix}
\alpha&\beta\\\gamma&\delta
\end{smallmatrix}
\right) \in \operatorname{GL}_2^+(\Z)$, is given by the change of variables
\begin{align*}
\left(Q\circ\left(\begin{matrix}
\alpha&\beta\\\gamma&\delta
\end{matrix}
\right)\right)(x,y)  = Q(x',y') = Q(\alpha x+\beta y,\gamma x+\delta y),
\end{align*}
so 
\begin{align}\label{eqn:trans-matrix-quad-forms}
\left(\begin{matrix}
x'\\y'
\end{matrix}
\right) = \left(\begin{matrix}
\alpha&\beta\\\gamma&\delta
\end{matrix}
\right) \left(\begin{matrix}
x\\y
\end{matrix}
\right).
\end{align}

Some convenient choices for the Atkin-Lehner matrices $ W_d $ in level $ 6 $  
\cite[p. 16]{Andersen2017-Singinv}
are 
\begin{gather*}
W_2 = \left(\begin{matrix}
2&-1\\6&-2
\end{matrix}
\right),\quad
W_3 = \left(\begin{matrix}
3&1\\6&3
\end{matrix}
\right),\quad
W_6=\left(\begin{matrix}
0&-1\\6&0
\end{matrix}
\right).
\end{gather*}
We normalize the Atkin-Lehner matrices when they act on quadratic forms by defining 
\begin{align}
\hat{W}_d:=\frac{1}{\sqrt{d}}W_d.
\end{align} 
If $ Q = [6a,b,c]\in \mathscr{Q} $ with $ a,b,c\in\Z $  then the $ \hat{W}_d  $ act on $ Q $ by 
\begin{align}\label{eqn:Wd on Q}
[6a,b,c]\circ \hat{W}_d = \begin{dcases}
[6a,b,c],&d=1,\\
[12a+6b+18c,-5b-12(a+c),3a+b+2c],&d=2,\\
[18a+6b+12c,5b+12(a+c),2a+b+3c],&d=3,\\
[6c,-b,a],&d=6.
\end{dcases}
\end{align}
Let $ I $ be the identity of $ \slz $ and note that $ -I $ has the same action as $ I $. Define 
\begin{align}\label{eqn:Gammaplus}
\Gamma^+=\left\langle\left(\Gamma_0(6)/\{\pm I\}\right)\cup\{\hat{W}_d:d\mid 6\}\right\rangle
\end{align} to be the group generated by $ \Gamma_0(6)/\{\pm I\} $ and the $ \hat{W}_d $,   
with the right action on $ \mathscr{Q} $ as above\footnote{Note that one can also take a left action of $ \Gamma^+ $ on quadratic forms instead, as is done in \cite[\S 3]{Andersen2017-Singinv}. If we denote Andersen's action by $ \gamma\cdot Q(x,y) $ for $ \gamma\in \Gamma^+ $ then it is related to our own by $ \gamma^{-1}\cdot Q(x,y) = \left(Q\circ \gamma\right)(x,y); $ however, for consistency with \cite{Kohnen-1985} we keep the right action. In \cref{sec: cycle-integrals} we will choose an orientation for the cycle integrals which is invariant under $ \Gamma^+ $.}. \par

Note that $\mathscr{Q}_{D,6}$ can be decomposed into the subsets $\mathscr{Q}^{(t)}_{D,6}$, where
\begin{align}\label{eqn: Qt quad defined}
\mathscr{Q}^{(t)}_{D,6}:=\{Q = [*,b,*]: \operatorname{disc} Q = D,\,6\mid Q(1,0),\,b\equiv t\pmod{12}\}.
\end{align}
$ \mathscr{Q}^{(t)}_{D,6} $ is stable under the action of $ \Gamma_0(6) $.
Let $ D$ be an odd fundamental discriminant with $ D\equiv 1\pmod{24} $. If $ (t,6)=1 $ then $ \mathscr{Q}^{(t)}_{D,6}/\Gamma_0(6) $ is a finite group under  Gaussian composition \cite[p. 16]{Andersen2017-Singinv}.

For $ D $ an odd fundamental discriminant the genus character $ \omega_D $ on $ Q = [a,b,c] $ with $ D\mid (b^2-4ac) $  is given by
\begin{align}\label{eqn:omega_D}
\omega_D(Q) = \begin{dcases}
    0&\text{ if }(a,b,c,D)>1,\\
    \charup{D}{d}&\text{ if }{(a,b,c,D)}=1\text{, }Q(x,y)=d\text{ for some }x,y\in\Z\text{, and }(d,D)=1.
\end{dcases}
\end{align}
As remarked by Kohnen this is a well-defined invariant of the equivalence classes of quadratic forms with discriminant divisible by $ D $,  modulo $ \slz $ \cite[{p. 238}]{Kohnen-1985}. 


\section{Poincar\'e Series for Powers of the Eta Multiplier}\label{sec:Poincare}

Let 
\begin{gather*}\Gamma_\infty:=\left\langle T,-I\right\rangle \leq  \slz \end{gather*}  
be the stabilizer of the cusp $i\infty$ for the full modular group. For $\sigma\equiv r\pmod{24}$, $ \sigma>0 $,  define the $\sigma$th Poincar\'e series of weight $ \lambda+1/2 $ and level $ 1 $ by  
\begin{gather*}
    P_{\lambda+1/2,\sigma}(\tau)= P_\sigma(\tau):=\sum_{\gamma\in {\Gamma_\infty}\backslash{\slz}}\nu_\eta(\gamma)^{-r}\left(e \left(\frac{\sigma \tau}{24} \right)\Big|_{\lambda+1/2}\gamma\right).
\end{gather*}
If $ \sigma\not\equiv r\pmod{24} $ we take $ P_{\sigma}=0 $.

Define the generalized Kloosterman sum for the multiplier $ \nu_\eta^r $  by
\begin{align}\label{eqn:gen-Kloost-eta}
A_{\nu_\eta^r}(\sigma,n;c):=\sum_{\substack{d\pmod{c}\\(d,c)=1}}\nu_\eta^{-\sigma}\left(\left(\begin{matrix}
\overline d&*\\c&d
\end{matrix}
\right)\right)e \left(\frac{\sigma \overline d+n d}{24c} \right),
\end{align}
where the matrices $ \left(\begin{smallmatrix}
\overline d&*\\c&d
\end{smallmatrix}
\right)\in\slz $ and $\overline{d}d\equiv 1\pmod{c}$. This is independent of the coset representative for $ \overline d $.  Note that when {$m\equiv n\equiv 1\pmod{24}$} (which corresponds to the multiplier $\nu_\eta^1$) we have $A_{\nu_\eta}(m,n;c)=\sqrt{i}K \left(\frac{m-1}{24},\frac{n-1}{24};c  \right)$, where $K$ is defined in \cite[(1.11)]{Andersen2017-Singinv}. Pribitkin and Williams derived the Fourier expansions for Poincar\'e series. The following is adapted from \cite[Lemma 2]{PW-2019-Rademacher} when $ \sigma>0 $.

\begin{lemma}
   
Let $ (r,6)=1 $ with $ 1\leq r\leq 23 $ and let $ \lambda\geq 2 $ be an integer.   Suppose $ S_{\lambda+1/2}(\slz,\nu_\eta^r)\neq \{0\} $.  Let $\delta_{*,*}$ be the Kronecker delta function, $J_{k}$ be the $k$th $J$-Bessel function, and $A_{\nu_\eta^r}$ be the generalized Kloosterman sum defined in \eqref{eqn:gen-Kloost-eta}. Let $r \leq\sigma$ and $\sigma\equiv r\pmod{24}$. Then
\begin{gather*}
P_{\lambda+1/2,\sigma}(\tau) =P_\sigma(\tau)= \sum_{\substack{m\equiv r\pmod{24}\\m\geq r}}\alpha_\sigma(m)e \left(\frac{m\tau}{24} \right),
\end{gather*}
where
\begin{gather}\label{eqn:alpha-poincare}
\alpha_{\sigma}(m):=\delta_{\sigma,m}+2\pi i^{-\lambda-1/2}\sqrt{\frac{m}{\sigma}}^{\lambda-1/2}\sum_{c\geq 1}\frac{A_{\nu_\eta^r}(\sigma,m;c)}{c}J_{\lambda-1/2}\left(\frac{\pi \sqrt{\sigma m}}{6c} \right).
\end{gather}
Moreover, $ P_{\lambda+1/2,\sigma}\in S_{\lambda+1/2}(\slz,\nu_\eta^r) $. 
\end{lemma}

Using the standard unfolding procedure we can prove the following.

\begin{proposition} \label{prop:Prho-unfolding}
Let $ (r,6)=1 $ with $ 1\leq r\leq 23 $ and let $ \lambda\geq 2 $ be an integer. Let $0\neq F(\tau) = \sum_{m\equiv r\pmod{24}}b(m)e \left(\tfrac{m\tau}{24}\right)\in S_{\lambda+1/2}(\slz,\nu_\eta^{r})\neq \{0\} $. Then for $\sigma\equiv r\pmod{24}$ with $r\leq \sigma $ we have
\begin{gather}\label{eqn:unfold-coeff}
    \left\langle F,P_{\lambda+1/2,\sigma}\right\rangle = \left(\frac{\sigma\pi}{6} \right)^{1/2-\lambda}\Gamma(\lambda-1/2)b(\sigma).
\end{gather}

\end{proposition}

\begin{proof}
Note that for $\gamma = \abcd\in \slz$ we have 
\begin{gather*}e \left(\frac{\sigma \tau}{24}\right)\Big|_{\lambda+1/2}\gamma = (c\tau+d)^{-\lambda-1/2}e \left(\frac{\sigma\gamma\tau}{24}\right)\end{gather*}
and
\begin{gather*}
    \overline{F(\tau)} = \nu_\eta(\gamma)^{r}\overline{(c\tau+d)}^{-\lambda-1/2}\overline{F(\gamma\tau)}.
\end{gather*}   
From the definition of $P_{\sigma}$ we have
\begin{align*}
    \left\langle P_\sigma,F\right\rangle &= \int_{\mathfrak F}\overline{F(\tau)}\left(\sum_{\gamma\in {\Gamma_\infty}\backslash{\slz}}\nu_\eta(\gamma)^{-r}\left(e \left(\frac{\sigma \tau}{24} \right)\Big|_{\lambda+1/2}\gamma\right)\right)(\Im(\tau))^{\lambda+1/2}d\mu\\
&=\int_{\mathfrak F}\sum_{\gamma=\left(\begin{smallmatrix}
*&*\\c&d
\end{smallmatrix}
\right)\in {\Gamma_\infty}\backslash{\slz}}\abs{c\tau+d}^{-2\lambda-1}\overline{F(\gamma\tau)}e \left(\frac{\sigma\gamma\tau}{24} \right)(\Im(\tau))^{\lambda+1/2}d\mu.
\end{align*}
Since $ \Im(\gamma \tau) = \tfrac{\Im (\tau)}{\abs{c\tau+d}^2}$ we see $ \abs{c\tau+d}^{-2\lambda-1} = \left(\tfrac{\Im (\gamma \tau)}{\Im (\tau)}\right)^{\lambda+1/2}$, so 
\begin{align}
\left\langle P_\sigma,F\right\rangle & = \int_{\mathfrak F}\sum_{\gamma\in {\Gamma_\infty}\backslash{\slz}}\overline{F(\gamma\tau)}e \left(\frac{\sigma\gamma\tau}{24} \right)(\Im(\gamma\tau))^{\lambda+1/2}d\mu .\nonumber 
\end{align}

 The integrands decay exponentially because $ P_\sigma $ and $ F $  are cusp forms, so we can exchange integration and summation.  
Hence
\begin{align}
\langle P_\sigma,F\rangle &= \int_0^\infty y^{\lambda-3/2}e^{\frac{-2\pi \sigma y}{24} }I(y)dy,
\end{align}
where $ I(y) = \int_0^1 \overline{F(x+iy)}e \left(\frac{\sigma x}{24} \right) dx$. 
We have
\begin{align*}
    I(y)&=\int_0^1 \overline{F(x+iy)}e \left(\frac{\sigma x}{24}\right)dx 
    \\
    &=\sum_{m\equiv r\pmod{24}}\overline{b(m)}e \left(\frac{miy}{24} \right)\int_0^1 e \left(\frac{x(\sigma-m)}{24} \right)dx
    \\ 
   &= \overline{b(\sigma)}e^{\frac{-2\pi \sigma y}{24} }.
\end{align*}
Therefore
\begin{align*}
    \left\langle P_\sigma,F\right\rangle &= \overline{b(\sigma)}\int_0^\infty y^{\lambda-3/2}e^{\frac{-4\pi\sigma y}{24}}dy\\
    &=\left(\frac{\sigma\pi}{6} \right)^{1/2-\lambda}\Gamma(\lambda-1/2)\overline{b(\sigma)}.
\end{align*} 
 \cref{eqn:unfold-coeff} follows from the fact that the Petersson product is Hermitian.

\end{proof}

\section{Representations of the Kernel Function}\label{sec: Repn of Kernel}

\subsection{Poincar\'e Representation}
Our first objective is to compute the Shimura maps in terms of Poincar\'e series, using the Petersson inner product.  
 The following proposition is a first step in proving \cref{thm: kernel-function}.

\begin{proposition}\label{prop:Kernel-Poincare}
Let $ (r,6)=1 $ with $ 1\leq r\leq 23 $, let $ \lambda\geq 2 $ be an integer, and let $F(\tau) = \sum_{m\equiv r\pmod{24}} a(m)e \left(\frac{m\tau}{24} \right)\in S_{\lambda+1/2}(\slz,\nu_\eta^r)$. Let $D$ be a fundamental discriminant with $ (-1)^\lambda D=\abs{D}\equiv r\pmod{24} $  and define $\Omega_{\lambda,r,D}^\circ(z;\tau)$ for $z,\tau\in\mathbb H$ by 

\begin{multline}\label{eqn:Omegacirc}
\Omega^\circ_{\lambda,r,D}(z;\tau) = \Omega_D^\circ(z;\tau) := \left(\frac{\pi\abs{D}}{6} \right)^{\lambda-1/2}\frac{1}{\Gamma(\lambda-1/2)} \\
\times\sum_{n\geq 1}n^{\lambda-1}
\left(
\sum_{\substack{d\mid n}}
\left(
n/d\right)^\lambda \chardown{D}{d} \charup{12}{n/d}P_{\tfrac{\abs{D}n^2}{d^2}}(\tau)\right)
e \left(nz\right).
\end{multline}
Then $ \Omega_{D}^\circ(z;\tau) \in S_{\lambda+1/2}(\slz,\nu_\eta^r) $ for fixed $ z\in\mathbb H $  and 
\begin{gather*}
\Sh_{\abs{D}}(F)(z) = \left\langle F(\cdot),\Omega_{D}^\circ (-\overline{z};\cdot)\right\rangle.
\end{gather*}

\end{proposition}

 \begin{proof}
We first show the series in \eqref{eqn:Omegacirc} converges absolutely. We have the bound  $ J_\nu(x)\ll_\nu x^\nu $ for $ \nu>-1/2 $ from \cite[8.4.11.4]{gradshteinTableIntegralsSeries2015}. This, together with $ \abs{A_{\nu_\eta^r}(\sigma,m;c)} \leq c$, shows the coefficients  $ \alpha_{\frac{\abs{D}n^2}{d^2}}(m) $ from \eqref{eqn:alpha-poincare} are uniformly polynomially bounded in $ m $ and $ n $. Then
\begin{multline*}
\sum_{n\geq 1}n^{\lambda-1}
\left(
\sum_{\substack{d\mid n}}
\left(
n/d\right)^\lambda\abs{\chardown{D}{d} \charup{12}{n/d}
P_{\frac{\abs{D}n^2}{d^2}}(\tau)
}\right)
\abs{e \left(nz\right)}
\\
\ll_{\lambda,D}
\sum_{n\geq 1}n^{\lambda-1}
\left(
\sum_{\substack{d\mid n}}
\left(
n/d\right)^{\lambda}\sum_{\substack{m\geq r\\m\equiv r\pmod{24}}}\abs{\alpha_\frac{\abs{D}n^2}{d^2}(m)e \left(\frac{m\tau}{24} \right)} 
\right)
\abs{e \left(nz\right)},
\end{multline*}
 converges for $ \tau,z $ in $ \mathbb{H} $. Hence $ \Omega_{D}^\circ(z;\tau)\in S_{\lambda+1/2}(\slz,\nu_\eta^r) $ for fixed $ z $ because $ {P_{\frac{\abs{D}n^2}{d^2}}(\tau)\in S_{\lambda+1/2}(\slz,\nu_\eta^r)} $. 

Computing the inner product gives
\begin{multline*}
\langle F(\cdot),\Omega_D^\circ(-\overline{z},\cdot)\rangle 
= 
\left(\frac{\pi\abs{D}}{6} \right)^{\lambda-1/2}\frac{1}{\Gamma(\lambda-1/2)}
\\
\times \sum_{n\geq 1}n^{\lambda-1}
\left(
\sum_{\substack{d\mid n}}
\left(
n/d\right)^\lambda \chardown{D}{d} \charup{12}{n/d}\langle 
   F(\tau), P_{\tfrac{\abs{D}n^2}{d^2}}(\tau)
\rangle
    \right)
\overline{e \left(-n\overline{z}\right)},
\end{multline*}
and applying \cref{prop:Prho-unfolding} gives
\begin{align*}
\langle F(\cdot),\Omega_D^\circ (-\overline{z},\cdot)\rangle 
= \sum_{n\geq 1}\left(\sum_{d\mid n}d^{\lambda-1}\charup{d}{D}\charup{12}{n/d}a \left(\frac{\abs{D}n^2}{d^2} \right)\right)e(nz).
\end{align*}
Using \eqref{eqn:AAD-Shim} shows that $ \langle F(\cdot),\Omega_D^\circ(-\overline{z},\cdot)\rangle  = \Sh_{\abs{D}}(F)(z) $, as desired.

 \end{proof}

\subsection{Representation in Terms of Quadratic Forms}\label{subsec:repn kernel quad forms} 
The connection between $ \Sh_{\abs{D}} $ and quadratic forms is given through the functions  $ g_{\lambda,6}(z;D,D') $ from \eqref{eqn:g-quad-form-defn}. To prove \cref{thm: kernel-function} we next need to show $ g_{\lambda,6}(z;D,D')\in S_{2\lambda}(\Gamma_0(6)) $.

 \begin{proposition}\label{prop:glambda-modular}
    Let $ \lambda\geq 2 $, $ D$ be a fundamental discriminant, and $DD' $ be a discriminant with $ DD'>0 $ and $ (DD',6)=1 $.   Then $g_{\lambda,6}(z;D,D')\in S_{2\lambda}(\Gamma_0(6)) $.   
   
 \end{proposition}
 \begin{proof}
   We first note that the series representation for $ g_{\lambda,6} $  converges absolutely and uniformly on compacta.
 Let $ Q=[a,b,c]\in \mathscr{Q}_{DD',6}^{(t)} $ and $ \left(\begin{smallmatrix}
\alpha&\beta\\\gamma&\delta
\end{smallmatrix}
\right)\in\slz $. We have
\begin{align}\label{eqn:Qlambda-action}
Q(z,1)^{-\lambda}\Big|_{2\lambda}\left(\begin{matrix}
\alpha&\beta\\\gamma&\delta
\end{matrix}
\right)&=
\nonumber 
\frac{(\gamma z+\delta)^{-2\lambda}}{\left(a \left(\frac{\alpha z+\beta}{\gamma z+\delta} \right)^2+b \left(\frac{\alpha z+\beta}{\gamma z+\delta} \right)+c\right)^\lambda}
\\
\nonumber
&=\left(a(\alpha z+\beta)^2+b(\alpha z+\beta)(\gamma z+\delta)+c(\gamma z+\delta)^2\right)^{-\lambda}\\
&=Q'(z,1)^{-\lambda}.
\end{align}
For brevity write $ \tilde{\omega}_D([a,b,c]) = \charup{12}{b}\omega_D([a,b,c]) $.
A short computation shows that $ {\left(\begin{smallmatrix}
\alpha&\beta\\\gamma&\delta
\end{smallmatrix}
\right)\in\Gamma_0(6)} $ acts as a bijection on $ \mathscr{Q}_{DD',6}^{(t)}$. Moreover, $ \omega_D(Q) $ is $ \slz$-invariant on $ \mathscr{Q}_{DD',6} $ \cite[p. 238]{Kohnen-1985}  so $ \tilde\omega_D(Q)=\tilde\omega_D \left(Q\circ \left(\begin{smallmatrix}
\alpha&\beta\\\gamma&\delta
\end{smallmatrix}
\right)\right)$. Hence, the action of $ \left(\begin{smallmatrix}
\alpha&\beta\\\gamma&\delta
\end{smallmatrix}
\right)\in\Gamma_0(6) $ permutes the summands of $ g_{\lambda,6}, $ and \eqref{eqn:Qlambda-action}, together with invariance of $ \tilde\omega_D $, shows these functions satisfy the modularity condition
$
 g_{\lambda,6}(z;D,D')\Big|_{2\lambda}\left(\begin{matrix}
\alpha&\beta\\\gamma&\delta
\end{matrix}
\right) = g_{\lambda,6}(z;D,D')
$
 for $ \left(\begin{smallmatrix}
\alpha&\beta\\\gamma&\delta
\end{smallmatrix}
\right)\in \Gamma_0(6) $. 
 
We now prove that $ g_{\lambda,6} $ vanishes at the cusps. Let $ \left(\begin{smallmatrix}
\alpha&\beta\\\gamma&\delta
\end{smallmatrix}
\right)\in\slz $. Then we have

\begin{align*}
\abs{g_{\lambda,6}(z;D,D')\Big|_{2\lambda}\left(\begin{matrix}
\alpha&\beta\\\gamma&\delta
\end{matrix}
\right)}\leq 
S_0+\sum_{
    \substack{Q\in \mathscr{Q}_{DD',6}\\
    Q'=\left[A,B,\frac{B^2-DD'}{4A} \right]=Q\circ \gamma\\
    A\neq 0}
}\abs{Az^2+Bz+\frac{B^2-DD'}{4A} }^{-\lambda},
\end{align*}
where $ S_0 = \sum_{Q'=[0,B,C]} \abs{Bz+C}^{-\lambda}$ with the sum taken over those $ Q\in\mathscr{Q}_{DD',6} $ with $ {Q'=Q\circ \left(\begin{smallmatrix}
\alpha&\beta\\\gamma&\delta
\end{smallmatrix}
\right)=[0,*,*]} $ having discriminant $ DD' $. 

The sum $ S_0 $ vanishes by comparison:
\begin{align*}
S_0&\leq \sum_{\substack{B\in\Z\\B^2=DD'}}\sum_{C\in\Z}\abs{Bz+C}^{-\lambda}
\ll_{DD',\lambda}\int_0^\infty \frac{du}{(u^2+y^2)^{\lambda/2}}
\ll_{DD',\lambda}y^{1-\lambda},
\end{align*}
where in the last comparison we substitute $ ty=u $. Hence $ S_0 $ vanishes as $ y\to\infty$. For the other sum, write $ \abs{Az^2+Bz+C} = \abs y^2\abs{A(\frac{x}{y}+i+\frac{B}{2Ay})^2-\frac{DD'}{4Ay^2}}\asymp_{DD,\lambda}\abs{y}^2 \abs{A(\frac{x}{y}+i+\frac{B}{2ay})^2} $ as $ y\to\infty $. Then we have
\begin{align*}
\sum_{
    \substack{Q\in \mathscr{Q}_{DD',6}\\
    Q'=\left[A,B,\frac{B^2-DD'}{4A} \right]=Q\circ \gamma\\
    A\neq 0}
}\abs{Az^2+Bz+\frac{B^2-DD'}{4A} }^{-\lambda}
&\ll_{DD',\lambda} y^{-2\lambda}\sum_{A\in\Z\setminus\{0\}}\abs{A}^{-\lambda}\sum_{\substack{B\in\Z\\ B^2\equiv DD'\pmod{4A}}}\abs{\frac{x}{y}+\frac{B}{2Ay}+i}^{-2\lambda}.
\end{align*}
Write $ B=t+2nA $ for $ t $ a residue modulo $ 2A $ and $ n\in\Z $. Then we have
    \begin{align*}
\sum_{
    \substack{Q\in \mathscr{Q}_{DD',6}\\
    Q'=\left[A,B,\frac{B^2-DD'}{4A} \right]=Q\circ \gamma\\
    A\neq 0}
}\abs{Az^2+Bz+\frac{B^2-DD'}{4A} }^{-\lambda}
&\ll_{DD',\lambda}
y^{-2\lambda}\sum_{A\in\Z\setminus \{0\}}\abs{A}^{-\lambda}\sum_{\substack{t\pmod{2A}\\t^2\equiv DD'\pmod{4A}}}\sum_{n\in\Z}\abs{u+\frac{n}{y}}^{-2\lambda},
\end{align*}
where $ u = x/y+t/2Ay +i $. Note $ \int_{-\infty}^\infty \abs{u+v/y}^{-2\lambda}dv\ll_\lambda y $ for $ \lambda >1$ and $ \Im(u)=1 $, so $ \sum_{n\in\Z}\abs{u+\frac{n}{y}}^{-2\lambda}\ll_{\lambda}y $. Because  $ {\sum_{\substack{t\pmod{2A}\\t^2\equiv DD'\pmod{4A}}}1\ll_{DD',\varepsilon}\abs{A}^\varepsilon} $ for all $ \varepsilon>0 $ by \cite[p.~236]{Koh-Est-Four-Siegel93}, we have
\begin{align*}
\sum_{
    \substack{Q\in \mathscr{Q}_{DD',6}\\
    Q'=\left[A,B,\frac{B^2-DD'}{4A} \right]=Q\circ \gamma\\
    A\neq 0}
}\abs{Az^2+Bz+\frac{B^2-DD'}{4A} }^{-\lambda}\ll_{DD',\lambda,\varepsilon}y^{1-2\lambda}
\end{align*}
as $ y\to\infty $. Combined with the estimate for $ S_0 $ we have that $ \abs{g_{\lambda,6}(z;D,D')\Big|_{2\lambda}\left(\begin{smallmatrix}
\alpha&\beta\\\gamma&\delta
\end{smallmatrix}
\right)} \to 0$ as $ y\to \infty $, and $ g_{\lambda,6} $ is a cusp form.

 \end{proof}
 
 We need to compute the Fourier expansion of $g_{\lambda,6}$ in order to prove \cref{thm: kernel-function}. There is a subtlety with the character, so we present the derivation here, though it is a modification of \cite[Prop. 2]{Kohnen-1985}. This Fourier expansion is expressed in terms of a twisted quadratic Weyl sum $ S. $ Let $ a>0 $ with $ 6\mid a $, $ D $ be a fundamental discriminant, and let $DD'$ be a discriminant with $DD'>0$ and $ (DD',6)=1 $. Define
\begin{align}\label{eqn:quad-SA-weyl-sum}
S_{a,D,D'}(DD',n):=\sum_{\substack{b\pmod{2a}\\b^2\equiv DD'\pmod{4a}}}\jacsym{12}{b}\omega_D \left(a,b,\frac{b^2-DD'}{4a} \right)e \left(\frac{nb}{2a} \right).
\end{align}

\begin{remark}
If $ \chi:\Z\to \C $ is a Dirichlet character then we extend $ \chi $ by $ \chi(x) = 0 $ for $ x\not\in \Z $. 
\end{remark}

 \begin{proposition}\label{prop:four-exp-g-full-quad}
 Let $ \lambda\geq 2 $. Let $ D$ be a fundamental discriminant with $ (-1)^\lambda D>0 $, $ m\geq 1 $, and $ \abs{D}m $ a discriminant coprime to $ 6 $. Let $ S $ be the twisted quadratic Weyl sum \eqref{eqn:quad-SA-weyl-sum}. Then
 \begin{align*}
 g_{\lambda,6}(z;D,(-1)^\lambda m ):=\sum_{\substack{Q=[*,b,*]\in \mathscr{Q}_{\abs{D}m,6}}}\charup{12}{b}\omega_D(Q)Q(z,1)^{-\lambda} = \sum_{n\geq 1}c_{\lambda,6}(n,D,(-1)^\lambda m)e(nz)
 \end{align*}
where
\begin{multline}\label{eqn:fourier-exp-g-full-quad}
c_{\lambda,6}(n,D,(-1)^\lambda m)\\
:=\frac{2(-2\pi)^\lambda}{\Gamma(\lambda)}\left(\slfrac{n^2}{\abs{D}m} \right)^{\frac{\lambda-1}{2}}\Bigg[
    (-1)^{\lfloor\frac{\lambda+1}{2}\rfloor}\charup{12}{\abs{D}}\charup{D}{\slfrac{n}{\sqrt{\frac{m}{\abs{D}}}}}\charup{12}{\sqrt{\slfrac{m}{\abs{D}}}}\abs{D}^{-1/2}\\
    +\pi\sqrt{2}\left(\slfrac{n^2}{\abs{D}m} \right)^{1/4}\sum_{\substack{a\geq 1\\6\mid a}}a^{-1/2}S_{a,D,(-1)^\lambda m}(\abs{D}m,n)J_{\lambda-1/2}\left(\frac{\pi n\sqrt{\abs{D}m}}{a}\right)
\Bigg].
\end{multline}

 \end{proposition}
 \begin{proof}
    Because $ (-1)^\lambda D>0 $ we have $ \charup{12}{-b}\omega_D(-a,-b,-c)  =(-1)^{\lambda} \charup{12}{b}\omega_D(a,b,c)$. Hence, for fixed $D,m,\lambda $, we can write  
 \begin{align*}
g_{\lambda,6}(z;D,(-1)^\lambda m) = 2 \left(g^0(z)+\sum_{\substack{a\geq 1\\6\mid a}}g^a(z)\right),
 \end{align*}
 where
 \begin{gather*}
 g^0(z):=\sum_{\substack{b>0\\b^2=\abs{D}m\\c\in\mathbb{Z}}}\charup{12}{b}\omega_D(0,b,c)(bz+c)^{-\lambda},\hspace{2em}\text{and}\\
 g^a(z):=\sum_{\substack{b\in \mathbb{Z}\\b^2\equiv \abs{D}m\pmod{4a}}}\charup{12}{b}\omega_D \left(a,b,\frac{b^2-\abs{D}m}{4a} \right)\left(az^2+bz+\frac{b^2-\abs{D}m}{4a} \right)^{-\lambda}.
 \end{gather*}

 For $g^0$ we have the discriminant $ \operatorname{disc}([0,b,c])=b^2=\abs{D}m $. Now $D\mid b$ since $ D $ is an odd fundamental discriminant, so $m=\abs{D}u^2$ for some $ u\in\Z_{>0}$. Hence
 \begin{align*}
 g^0(z) = \sum_{\substack{b>0\\b^2=\abs{D}^2u^2\\c\in\Z}}\charup{12}{b}\omega_D(0,b,c)(bz+c)^{-\lambda}.
 \end{align*}
 Note that $ \charup{12}{b} \omega_D(0,b,c)=\charup{12}{Du}\charup{D}{c}. $ Splitting the $c$ into residues modulo $D$ gives
 \begin{align*}
 g^0(z)
 &=\charup{12}{Du}\sum_{t\pmod{D}}\charup{D}{t}\abs{D}^{-\lambda}\sum_{n\in \Z}\left(uz+\frac{t}{\abs{D}}+n \right)^{-\lambda}. 
 \end{align*}
 The sum is exactly the one that appears on \cite[p. 247]{Kohnen-1985}, which is computed using Lipschitz's formula and evaluation of a Gauss sum. Hence, using this evaluation (and correcting a minor typo in \cite[(19)]{Kohnen-1985}), we obtain
 \begin{align*}
 g^0(z) = \charup{12}{Du}(-1)^{\lfloor\frac{\lambda}{2}\rfloor}\abs{D}^{-\lambda+1/2}\frac{(2\pi)^\lambda}{\Gamma(\lambda)}\sum_{n\geq1}\charup{D}{n}n^{\lambda-1}e(nuz).
 \end{align*}
 Reindexing the sum and recalling that $ u=\sqrt{m/\abs D} $, we see that the contribution to the Fourier coefficient of $e(nz)$ from $g^0$ is
 \begin{align*}
   \charup{12}{D}\charup{12}{\sqrt{m/\abs D}}(-1)^{\lfloor \frac{\lambda}{2}\rfloor }\abs{D}^{-\lambda+1/2}\frac{(2\pi)^\lambda}{\Gamma(\lambda)}\charup{D}{n\Big/\sqrt{\frac{m}{\abs D}}}\left(n\Big/\sqrt{\frac{m}{\abs D}}\right)^{\lambda-1}  ,
 \end{align*}
 bearing in mind our definition of $ \charup{D}{x}=0$ if $ x\not\in\Z $  and similarly for $ \charup{12}{\cdot}. $ This simplifies to the first term of \eqref{eqn:fourier-exp-g-full-quad} because $ (-1)^\lambda (-1)^{\lfloor\frac{\lambda+1}{2}\rfloor} = (-1)^{\lfloor\frac{\lambda}{2}\rfloor} $.

We now consider $ g^a, $ where $ 6\mid a $. We have
 \begin{align*}
 g^a(z) = \sum_{\substack{b\pmod{2a}\\b^2\equiv \abs{D}m\pmod{4a}}}\charup{12}{b}\omega_D \left(a,b,\frac{b^2-\abs{D}m}{4a} \right)\sum_{n\in \Z}\left(a(z+n)^2+b(z+n)+\frac{b^2-\abs{D}m}{4a} \right)^{-\lambda}
 \end{align*}
 because 
 \[ a(z+n)^2+b(z+n)+\frac{b^2-\abs{D}m}{4a} = az^2+(b+2an)z+\frac{(b+2an)^2-\abs D m}{4a}   .
 \] 
 The inner sum also appears in \cite[p. 247]{Kohnen-1985}, and has $ n $th Fourier coefficient given by 
 \begin{align*}
 \frac{(-1)^\lambda 2^{\lambda+1/2}\pi^{\lambda+1}n^{\lambda-1/2}}{(\abs{D}m)^{\lambda/2-1/4}\Gamma(\lambda)\sqrt{a}}J_{\lambda-1/2}\left(\frac{\pi n\sqrt{\abs{D}m}}{a} \right)e(nb/2a). 
 \end{align*}
 Hence summing over $ b\pmod{2a} $ with $ b^2\equiv \abs{D}m\pmod{4a} $ gives the quadratic Weyl sum $ S_{a,D,(-1)^\lambda m}(\abs{D}m,n) $  and the second term in the expansion of \eqref{eqn:fourier-exp-g-full-quad}.

 \end{proof}

\subsection{Proof of Equivalence}\label{subsec:proof-equiv-kernel}
\begin{remark} 
If $m$ and $n$ are coprime integers with odd parts $m_o$ and $n_o$, respectively, the two characters $ \jacsym{m}{n}\text{ and }\jacsym{n}{m} $ are related by the quadratic reciprocity relation \cite[(5.7)]{stromberg2013-weil-quad} given by

\begin{align}\label{eqn:quad-rec}
\charup{m}{n}\chardown{m}{n} = (m,n)_{\infty}(-1)^{\frac{(m_o-1)(n_o-1)}{4}},
\end{align}
where $ (m,n)_\infty $ is the Hilbert symbol at infinity, which is $ -1 $ if both $ m,n<0 $ and $ 1 $ otherwise.   

\end{remark}

From the Fourier expansions of $ P_{\lambda+1/2,\sigma} $ and $ g_{\lambda,6} $ (\eqref{eqn:alpha-poincare} and \eqref{eqn:fourier-exp-g-full-quad}, respectively) we see that we will need a relation between Kloosterman sums and quadratic Weyl sums. The necessary identity is the main result of concurrent work by the author, proven by writing $ \nu_\eta^r $ as a sum of matrix coefficient functions for the Weil representation attached to a finite quadratic module and modifying the methods of \cite{AAW2025-GenK-Weyl-Bound}.
\begin{theorem}\cite[Thm. 1.1]{Williams2026Powers}\label{thm: relate-Kloost-Weyl}
Let $ 1\leq r\leq 23 $ be such that $ (r,6)=1 $, suppose that $ D $ is a fundamental discriminant with $ (-1)^{\frac{r-1}{2} } D=\abs{D}\equiv r\pmod{24} $, and suppose that $ m $ is positive with  $ m \equiv r\pmod{24} $. Then for integral $ a,n>0 $ we have  
\begin{multline}\label{eqn:kloost-weyl identity}
S_{6a,D,(-1)^{\frac{r-1}{2} }m}(\abs{D}m,n) = \varepsilon_r i^{-1/2}\jacsym{12}{r}2 \sqrt{3}\\
\times \sum_{d\mid (n,a)}\sqrt{d/a}\jacsym{d}{D}\jacsym{12}{n/d}A_{\nu_\eta^r}\left(\frac{\abs{D}n^2}{d^2},m;\frac{a}{d} \right).
\end{multline}

\end{theorem}

Note that, assuming the condition of \eqref{eqn:consistency}, we have $ (-1)^\lambda=(-1)^{\frac{r-1}{2} } $. Using this result we can prove \cref{thm: kernel-function}.

\begin{proof}[Proof of \cref{thm: kernel-function}.] Recall that $ \Omega_{\lambda,r,D}^\circ $ is given in \eqref{eqn:Omegacirc} and that $ \Omega_{\lambda,r,D} $ is defined by  \eqref{eqn: Kernel-quad-repn}.
    In view of \cref{prop:Kernel-Poincare} it suffices to prove that $ \Omega_{\lambda,r,D}(z;\tau)\in S_{2\lambda}(\Gamma_0(6)) $ for fixed $ \tau $  and that
    \begin{align}
    \Omega_{\lambda,r,D}^\circ(z;\tau) = \Omega_{\lambda,r,D}(z;\tau).
    \end{align}
    As before we suppress $ \lambda $ and $ r $ in our notation.  
    We proceed, as in \cite[p. 257]{Kohnen-1985}, by comparing the two representations of $ \Omega_D $ coefficient-wise. The series representations converge absolutely and uniformly on compacta. For $ \Omega_D^\circ $ this is part of the proof of \cref{prop:Kernel-Poincare}. 
    
    The Fourier coefficients $ c_{\lambda,6}(n,D,(-1)^\lambda m) $  of $ g_{\lambda,6}(z;D,(-1)^\lambda m) $ are polynomially bounded in $ n $ since $ g_{\lambda,6} $ is a cusp form. The coefficients are also polynomially bounded in $ m $, using \eqref{eqn:fourier-exp-g-full-quad}.    
Therefore the series \eqref{eqn: Kernel-quad-repn} converges uniformly and absolutely on compacta. Since  $ g_{\lambda,6}(z;D,(-1)^\lambda m)\in S_{2\lambda}(\Gamma_0(6))$  by \cref{prop:glambda-modular}, we have $ \Omega_D(z;\tau) \in S_{2\lambda}(\Gamma_0(6))$ for fixed $ \tau $.

    We express $ \Omega_D^\circ $ in terms of the Fourier coefficients $ \alpha_{\frac{\abs{D}n^2}{d^2}}(m) $ for $ P_{\frac{\abs{D}n^2}{d^2}} $  and $ \Omega_D $ in terms of the Fourier coefficients $ c_{\lambda,6} $  for $ g_{\lambda,6} $ using \eqref{eqn:Omegacirc} and \eqref{eqn: Kernel-quad-repn}, respectively. Comparing the Fourier coefficients of $ e(nz)e(m\tau/24) $  we see that it suffices to prove
\begin{multline}
\kappa_{\lambda,r,D}\,m^{\lambda-1/2}c_{\lambda,6}(n,D,(-1)^{\lambda}m) \\
=
\left(\frac{\pi\abs D}{6} \right)^{\lambda-1/2}\frac{n^{\lambda-1}}{\Gamma(\lambda-1/2)}\sum_{d\mid n}(n/d)^{\lambda}\charup{d}{D}\charup{12}{n/d}\alpha_{\frac{\abs{D}n^2}{d^2} }(m). 
\end{multline}

Note that $ (-1)^{\lfloor\frac{\lambda+1}{2}\rfloor}=i^{\lambda}\varepsilon_r $ when $\lambda$ and $r$ satisfy \eqref{eqn:consistency}. We can combine $ \Gamma $ and $ \beta $ factors using $ \frac{\beta(\lambda,-1/2)}{\Gamma(\lambda)} = \frac{-2\sqrt{\pi}}{\Gamma(\lambda-1/2)} $.
Using \eqref{eqn:kappa}, \eqref{eqn:alpha-poincare}, and \eqref{eqn:fourier-exp-g-full-quad}, we see that it suffices to show the following:
\begin{multline}\label{eqn: key-fourier-comp}
i^{\lambda}\varepsilon_r \charup{12}{r}m^{\lambda/2}
\Bigg[
i^\lambda \varepsilon_r \charup{12}{\abs D}\charup{D}{\slfrac{n}{\sqrt{m/\abs{D}}}}\charup{12}{\sqrt{\slfrac{m}{\abs{D}}}}\abs{D}^{-1/2}\\
+\pi\sqrt{2}\left(\slfrac{n^2}{\abs{D}m}\right)^{1/4}\sum_{a\geq 1}(6a)^{-1/2}S_{6a,D,(-1)^\lambda m}(\abs{D}m,n)J_{\lambda-1/2} \left(\frac{\pi n\sqrt{\abs{D}m}}{6a} \right)    
\Bigg]\\
=
\abs{D}^{\frac{\lambda-1}{2}}\sum_{d\mid n}(\slfrac{n}{d})^{\lambda}\charup{d}{D}\charup{12}{n/d}
\leavevmode\hfill
\\
\times 
\Bigg[
\delta_{\frac{\abs D n^2}{d^2},m}+2\pi i^{-\lambda-1/2}\sqrt{\frac{m}{\abs{D}n^2/d^2} }^{\lambda-1/2}\sum_{c\geq 1}\frac{A_{\nu_\eta^r} \left(\frac{\abs{D}n^2}{d^2},m;c \right)}{c}J_{\lambda-1/2}\left(\frac{\pi n \sqrt{\abs D m}}{6cd} \right)     
\Bigg] .
\end{multline}

If $ \sqrt{m/\abs{D}}\nmid n $ then the first term on the left hand side is zero, as the characters are zero on non-integers. Hence the first term on either side of \eqref{eqn: key-fourier-comp} is zero unless $ m = \abs{D}u^2 $ for some $ u \mid n $ with  $ (u,6)=1 $. The coprime condition occurs because of the $ \charup{12}{\cdot} $ symbol. Any $ u $ satisfying such a condition is a unique divisor of $ n $, and when such is the case we have
\begin{align*}
m^{\lambda/2}\abs{D}^{-1/2} = \abs{D}^{\frac{\lambda-1}{2}}(n/d)^\lambda ,\quad u=n/d.
\end{align*}
In this case the divisor sum collapses and equality of the first terms on either side will therefore follow from
\begin{align*}
\charup{D}{n/u}=\charup{n/u}{D}.
\end{align*}
This is true by quadratic reciprocity, \eqref{eqn:quad-rec}, since $ D $ is an odd fundamental discriminant.  

We now turn our attention to the second terms in \eqref{eqn: key-fourier-comp} (the sums over $ a $ and $ c $). We rewrite the sum over $ c $ as a sum over $ a $ by writing $ c = a/d $ for some $ a\in\Z_{>0} $; doing so and distributing, we see that equality for these terms follows from
 \begin{multline*}
 \varepsilon_r\charup{12}{r}\sum_{a\geq 1}\frac{S_{6a,D,(-1)^\lambda m}(\abs{D}m,n)}{\sqrt{a}}J_{\lambda-1/2}\left(\frac{\pi n \sqrt{\abs{D}m}}{6a} \right) \\
= 2\sqrt{3}(-1)^\lambda i^{-1/2}
\sum_{a\geq 1} \sum_{d\mid (n,a)}\charup{d}{D}\charup{12}{n/d}\frac{\sqrt{d}}{a}
A_{\nu_\eta^r}\left(\frac{\abs{D}n^2}{d^2},m;\frac{a}{d} \right)J_{\lambda-1/2}\left(\frac{\pi n \sqrt{\abs{D}m}}{6a} \right).   
 \end{multline*}
Using \eqref{eqn:kloost-weyl identity}, proven in \cite{Williams2026Powers}, together with the fact that $ (-1)^\lambda\varepsilon_r^{-1}=\varepsilon_r $ when $ \lambda,r $ satisfy \eqref{eqn:consistency}, concludes the proof that $ \Omega_D = \Omega_D^\circ $  by comparing summands.
  Since $ \Omega_D^\circ\in S_{\lambda+1/2}(\slz,\nu_\eta^r) $ and $ \Omega_D\in S_{2\lambda}(\Gamma_0(6)) $ this also proves \cref{thm: kernel-function}.

\end{proof}

We can now prove \cref{cor: shintani-property}.
 Recall the definition of the Shintani maps $ \Sh_{\abs D}^* $ from \eqref{eqn:shintani-defn}.

\begin{proof}[Proof of \cref{cor: shintani-property}]
Fix $ \lambda,r $. Let $ z=x+iy $ be the variable for $ f\in S_{2\lambda}^{\operatorname{new}}(\Gamma_0(6)) $ and $ \tau = X+iY $ the variable for $ F\in S_{\lambda+1/2}(\slz, \nu_\eta^r). $ Note $ \Omega_{\lambda,r,D}(-\overline z;\tau)=\overline{\Omega_{\lambda,r,D}(z;-\overline\tau)} $ because $ \kappa_{\lambda,r,D}\in \R $ with \eqref{eqn:consistency} and $ g_{\lambda,6}(-\overline z,D,(-1)^{\lambda}m) = \overline{g_{\lambda,6}(z,D,(-1)^\lambda m)} $.     Then \cref{thm: kernel-function} and \cref{prop:Kernel-Poincare} gives us
\begin{multline*}
\Big\langle \Sh_{\abs{D}}(F(\tau))(z),f(z)\Big\rangle = \Big\langle \langle F(\tau),\Omega_{\lambda,r,D}(-\overline z;
\tau)\rangle,f(z)\Big\rangle\\
=
\frac{1}{12}\int_{\mathfrak{F}_{\Gamma_0(6)}}\left(\int_{\mathfrak{F}}F(\tau)\overline{\Omega_{\lambda,r,D}(-\overline z;\tau)}Y^{\lambda+1/2}\frac{dXdY}{Y^2}\right)\overline{f(z)}y^{2\lambda}\frac{dxdy}{y^2}\\
=\frac{1}{12}\int_{\mathfrak{F}}F(\tau)\left(\int_{\mathfrak{F}_{\Gamma_0(6)}}\overline{f(z)\Omega_{\lambda,r,D}(-\overline{z};\tau)}y^{2\lambda}\frac{dxdy}{y^2}\right)Y^{\lambda+1/2}\frac{dXdY}{Y^{2}}.
\end{multline*}
Hence
\begin{gather*}
\Big\langle \Sh_{\abs{D}}(F(\tau))(z),f(z)\Big\rangle = \Big\langle F(\tau),\Sh_{\abs{D}}^{*}(f(z))(\tau)\Big\rangle.
\end{gather*}

\end{proof}

\section{Cycle Integrals}\label{sec: cycle-integrals} In this section we relate the Fourier coefficients of the Shintani map $ \Sh_{\abs D}^* $ to sums of cycle integrals, proving  \cref{thm: cycle-int-and-coeff}. We will also prove that these sums are related to central values of twisted $ L $-functions in \cref{prop: central-cycle-reln}, and at the end of this section prove \cref{thm:Main-KZ-formula}.\par

Cycle integrals are integrals of modular forms along geodesics in the hyperbolic geometry, with respect to a differential depending on the geodesic. Each quadratic form $ Q\in\mathscr{Q} $ with positive discriminant determines a geodesic in $ \mathbb H $. Define $ S_Q $ to be the geodesic in $ \mathbb{H}\cup\mathbb{R}\cup\{i\infty\} $ going either through the two roots of $ Q(z,1)=az^2+bz+c $ in $ \mathbb{R} $ if $ a\neq 0 $, or through the real root of $ Q(z,1)=bz+c $ and $ i\infty $. The orientation on $ S_Q $ is given as follows\footnote{We adopt this convention because it preserves the orientation of $ S_Q $ under the action of $ \Gamma^+ $ when $ Q $ has square discriminant, cf. \cite[p. 286--287]{Kohnen-Rmk-1988}, \cite[p. 240]{Kohnen-1985}.}:

\begin{itemize}
\item If $ a\neq 0 $ (so $ S_Q $ is not a vertical line) then $ S_Q $ is the directed arc from $\frac{-b-\sqrt{b^2-4ac}}{2a}  $ to $ \frac{-b+\sqrt{b^2-4ac}}{2a}   $,  along the curve $  a\abs{z}^2+b\Re(z) +c =0 $.
\item If $ a=0 $ and $ b<0 $  then $ S_Q $ is the line segment from $ -c/b $ to $ i\infty $.
\item If $ a=0 $ and $ b>0 $  then $ S_Q $ is the line segment from $ i\infty $ to $ -c/b $.   
\end{itemize}

Let $ D $ be a positive discriminant coprime to $ 6 $ and $ z_Q\in \R $ be a root of $ Q(z,1) $ for $Q\in \mathscr{Q}_{D,6} $. Let $ \gamma = \left(\begin{smallmatrix}
*&*\\\beta&\delta
\end{smallmatrix}
\right)\in \Gamma^+ $, then when $\beta z_Q+\delta\neq 0$ we have
\begin{align*}
Q\circ \gamma^{-1}(\gamma z_Q,1) = \frac{Q(z_Q,1)}{(\beta z_Q+\delta)^2}=0 .
\end{align*}
If $ \beta z_Q+\delta=0 $ then $ \gamma z_Q=\infty $, and the geodesic is a straight line. 
We define the differential on $ S_Q $ by
\begin{align}\label{eqn:dQdiff}
    d_{Q,\lambda}z:=Q(z,1)^\lambda\frac{dz}{Q(z,1)}.
\end{align}
Then  for $ \gamma\in \Gamma_0(6) $ and $ f\in S_{2\lambda}(\Gamma_0(6)) $  one can compute directly that
\begin{align}\label{eqn: diff-dQ-under-gamma}
f(z)d_{Q\circ \gamma,\lambda}z = f(\gamma z)d_{Q,\lambda}(\gamma z).
\end{align}
Define $ C_Q $ to be the image of $ S_Q $ under the quotient map $ \mathbb{H}\mapsto\,{\Gamma_0(6)}\backslash\mathbb{H} $. A cycle integral of $ f\in S_{2\lambda}(\Gamma_0(6)) $ is an integral of the form
\begin{align*}
\int_{C_Q}f(z)d_{Q,\lambda}z.
\end{align*}

Let $ Q' = Q\circ \gamma $ for $ \gamma\in\Gamma_0(6) $ and $ Q,Q'\in\mathscr{Q}_{DD',6} $.
From \eqref{eqn: diff-dQ-under-gamma} we have
\begin{align*}
\int_{C_{Q'}}f(z)d_{Q',\lambda}z &=\int_{C_{Q\circ \gamma}}f(\gamma z)d_{Q,\lambda}(\gamma z)=\int_{\gamma C_{Q\circ\gamma}}f(z)d_{Q,\lambda}z.
\end{align*}
Because $ \gamma S_{Q\circ \gamma} = S_Q$ we have
\begin{align}
\int_{C_Q}f(z)d_{Q,\lambda}z=\int_{C_{Q'}}f(z)d_{Q',\lambda}z.
\end{align}
Hence the cycle integral is a well-defined invariant of the equivalence class of $ Q $ with respect to $ \Gamma_0(6) $.\par

 Let $ \ell\mid 6 $, put $ z'=W_\ell z$, and suppose that $ S_Q,S_{Q\circ \hat {W}_\ell} $ are not vertical lines.  As $ z $ traverses between the roots of $ Q\circ \hat{W}_\ell $ we see that $ z' $ traverses between the roots of $ Q $. A similar fact is true when one of $ S_Q $ or $ S_{Q\circ\hat{W}_\ell} $ is a vertical line.   Using this, a short computation shows that if $ f\Big|_{2\lambda}W_\ell = \delta_\ell f $ with $ \delta_\ell=\pm 1 $, then we have 
 \begin{align}\label{eqn: invariant CQ under involution}
    \int_{C_Q\circ \hat{W}_\ell}f(z)d_{Q\circ \hat{W}_\ell,\lambda}(z) = \delta_\ell \int_{C_Q}f(z')d_{Q,\lambda}z'.
 \end{align}

\subsection{Cycle Integrals and the Shintani Map}\label{subsec:cycle-int-shintani}
We first relate a sum of cycle integrals to the modular forms $ g_{\lambda,6} $ from \eqref{eqn:g-quad-form-defn}, and then compute the Fourier coefficients of the Shintani lifts. 

\begin{proposition}\label{prop: cycle-int-petersson}
Let $ (r,6)=1 $ with $ 1\leq r\leq 23 $, let $ \lambda\geq 2 $ be an integer, and let $ D $ be a fundamental discriminant with $ (-1)^\lambda D=\abs{D}\equiv r\pmod{24} $. Let $ m\geq 0 $ with $ {m\equiv r\pmod{24}} $, and let $ \mathds{1}_\square $ be the indicator function for perfect squares.   Let $ f\in S_{2\lambda}(\Gamma_0(6)) $, and define $ s_{\lambda,6}(f;D,(-1)^\lambda m) $ by \eqref{eqn:s-cycle-sum}. Then 
\begin{align}\label{eqn:unfold-cycleint}
\langle f(\cdot),g_{\lambda,6}(\cdot;D,(-1)^\lambda m)\rangle = (-1)^{\mathds{1}_{\square}(\abs{D}m)}\frac{ 2^{-2\lambda+2}\pi(\abs{D}m)^{1/2-\lambda}}{[\slz:\Gamma_0(6)]}\binom{2\lambda-2}{\lambda-1}s_{\lambda,6}(f;D,(-1)^\lambda m).
\end{align}

\end{proposition}
This follows from \cite[Prop. 7]{Kohnen-1985}, taking care with the orientation on $ S_Q $.  

\begin{proof}[Proof]
    Let $ \mathfrak{A} $ be an equivalence class for the action of $ \Gamma_0(6)  $ on $ \mathscr{Q}_{\abs{D}m,6} $. Define   
\begin{align*}
f_{\lambda,6}(z;\mathfrak{A}):=\sum_{Q\in\mathfrak{A}}Q(z,1)^{-\lambda}.
\end{align*}
     Let $ Q $ be any element of $ \mathfrak{A} $ with stabilizer $ \Gamma_Q $ in $ \Gamma_0(6) $. Then writing $ z=x+iy $ and unfolding yields
     \begin{align*}
     \langle f(\cdot),f_{\lambda,6}(\cdot;\mathfrak{A}) \rangle&= \frac{1}{[\slz:\Gamma_0(6)]}\int_{\Gamma_0(6)\backslash \mathbb H}
     f(z)\sum_{\gamma\in \Gamma_Q\backslash \Gamma_0(6)}\left(Q\circ \gamma(\overline{z},1)\right)^{-\lambda} y^{2\lambda-2}dxdy
     \\
        &=\frac{1}{[\slz:\Gamma_0(6)]}\int_{\Gamma_Q\backslash \mathbb H}
     f(z)Q(\overline z,1)^{-\lambda} y^{2\lambda-2}dxdy.
     \end{align*}
    \cite[Prop. 7, p. 266]{Kohnen-1985} gives (after computing $ \int_0^\pi \sin^{2\lambda-2}\theta d\theta $)
     \begin{align}\label{eqn:all-cycle}
     \langle f(\cdot),f_{\lambda,6}(\cdot;\mathfrak{A})\rangle = \left(-1\right)^{\mathds{1}_{\square}(\abs{D}m)}\frac{ 2^{-2\lambda+2}\pi(\abs{D}m)^{1/2-\lambda}}{[\slz:\Gamma_0(6)]}\binom{2\lambda-2}{\lambda-1}\int_{C_Q}f(z)d_{Q,\lambda}z.
     \end{align}

Because $ \Gamma_0(6) $ preserves $ b \pmod{12}$ when acting on $ Q=[*,b,*] $ and $ \omega_D $ is constant on equivalence classes, we have
\begin{align*}
g_{\lambda,6}(z;D,(-1)^\lambda m) = \sum_{b\pmod{12}^\times}\charup{12}{b}\sum_{\substack{\mathfrak{A}\in{\Gamma_0(6)}\backslash{\mathscr{Q}_{\abs{D}m,6}}\\
[*,b,*]\in\mathfrak{A}
}}\omega_D(\mathfrak{A})f_{\lambda,6}(z;\mathfrak{A}).
\end{align*}
Then \eqref{eqn:unfold-cycleint} follows from the definition \eqref{eqn:s-cycle-sum} of $ s_{\lambda,6} $ and from \eqref{eqn:all-cycle}.
   
\end{proof}

We are now in a position to compute the Fourier coefficients of the Shintani lifts. 
\begin{proposition}\label{prop: cycle-int-shintani}Let $ (r,6)=1 $ with $ 1\leq r\leq 23 $, let $ \lambda\geq 2 $ be an integer, and let $ D $ be a fundamental discriminant with $ (-1)^\lambda D=\abs{D}\equiv r\pmod{24} $. Let $ f\in S_{2\lambda}^{\operatorname{new}}(\Gamma_0(6)) $, and write 
    \begin{align*}
    \Sh_{\abs{D}}^*(f) (\tau) = \sum_{\substack{m\geq r\\m\equiv r\pmod {24}}}C(f,D,(-1)^\lambda m)e \left(\tfrac{m\tau}{24}\right).
    \end{align*}
Then 
\begin{gather}\label{eqn:shintai-coeff-C}
C(f,D,(-1)^\lambda m) = \frac{(-1)^{\mathds{1}_{\square}(\abs{D}m)}i^{-\lambda}\varepsilon_r}{4\sqrt{6}(12)^\lambda}\charup{12}{r} s_{\lambda,6}(f,D,(-1)^\lambda m). 
\end{gather}

\end{proposition}

\begin{proof}
Let $ \kappa_{\lambda,r,D} $ be as in \cref{thm: kernel-function}. Then by \eqref{eqn: Kernel-quad-repn} we have
\begin{align*}
    \Sh_{\abs{D}}^*(f)(\tau)=
\langle f(\cdot),\Omega_{\lambda,r,D}(\cdot;-\overline{\tau})\rangle=
\kappa_{\lambda,r,D}\Bigg[
\sum_{m\equiv r\pmod{24}}m^{\lambda-1/2}\langle f(\cdot),g_{\lambda,6}(\cdot;D,(-1)^\lambda m)\rangle e \left(\frac{m\tau}{24} \right)    
\Bigg].
\end{align*}
The result follows from \cref{prop: cycle-int-petersson} using the fact that
\begin{align*}
\beta(\lambda,-1/2)\binom{2\lambda-2}{\lambda-1} = -2^{2\lambda-1}.
\end{align*}

\end{proof}

\subsection{Coefficients and Central Values}
The following proof is analogous to \cite[Thm. 3]{Kohnen-1985}.

\begin{proof}[Proof of \cref{thm: cycle-int-and-coeff}]
Let $ f $ be a newform in $ S_{2\lambda}(\Gamma_0(6),\delta_2(r),\delta_3(r)) $ with Hecke eigenvalue $ b(p) $ under $ T_p $   and let $ F\in S_{\lambda+1/2}(\slz,\nu_\eta^r) $ be a Hecke eigenform corresponding to $ f $. Let $ D $ be a fundamental discriminant with $ (-1)^{\lambda}D=\abs{D}\equiv r\pmod{24} $.

 For all primes $ p\geq 5 $ and all  $ G\in S_{\lambda+1/2}(\slz,\nu_\eta^r) $ we have
 \begin{align*}
 \left\langle \Sh_{\abs{D}}^*(f)\Big|T_{p^2},G\right\rangle
 &= \charup{12}{p}
 \left\langle f, \Sh_{\abs{D}}(G)\Big|T_{p} \right\rangle
 \\
 &=\charup{12}{p}
 \left\langle f \Big|T_p,\Sh_{\abs{D}}(G)\right\rangle\\
 &=
 \charup{12}{p}b(p)\left\langle \Sh_{\abs{D}}^*(f) ,G\right\rangle.
 \end{align*}
 Then $ \Sh_{\abs{D}}^*(f) = \alpha F $ for some $ \alpha\in\C $,  by multiplicity one for $ S_{\lambda+1/2}(\slz,\nu_\eta^r) $ (\cref{cor: mult-one-eta-mult}). 
Hence 
\begin{align*}
    \Sh_{\abs{D}}^*(f)(\tau) =\sum_{m\equiv r\pmod{24}}C(f,D,(-1)^\lambda m)e \left(\tfrac{m\tau}{24}\right) 
    = \alpha\sum_{m\equiv r\pmod{24}}a(m)e \left(\tfrac{m\tau}{24}\right)=\alpha F(\tau),
\end{align*}
    where $ C(f,D,(-1)^\lambda m) $ is as in \eqref{eqn:shintai-coeff-C}.  Using \cref{cor: shintani-property} yields
\begin{align*} 
C(f,D,(-1)^\lambda m)\langle F,F\rangle &= \alpha a(m)\langle F,F\rangle\\
&=a(m)\langle \Sh_{\abs{D}}^* (f),F\rangle\\
&=a(m)\langle f,\Sh_{\abs{D}} (F)\rangle\\
&=a(m)\overline{a(\abs D)}\langle f,f\rangle
\end{align*}
by \eqref{eqn: St scalar mult}. Substitution of \eqref{eqn:shintai-coeff-C} proves \cref{thm: cycle-int-and-coeff}.

\end{proof}

We need some facts about the genus character, which can be found in \cite[p. 243]{Kohnen-1985}. We first observe that for $ d\mid 6 $ we have 
\begin{align}
\omega_D(Q\circ \hat{W}_d) = \charup{D}{d}\omega_D(Q).
\end{align}
Let $ Q_u = [0,D,u]; $ then a complete class of representatives of $ \mathscr{Q}_{D^2,6} $ under $ \Gamma_0(6) $ is given by
\begin{align}\label{eqn: complete-class}
\{Q_u\circ \hat{W}_d:u\pmod{\abs D},\,d\mid 6\}.
\end{align}
For brevity we define $ \tilde\omega_D $ by
\begin{align*}
\tilde\omega_D([*,b,*]) = \charup{12}{b}\omega_D([*,b,*]).
\end{align*}
Using \eqref{eqn:Wd on Q} and the fact that $ \charup{12}{\pm 5b} = -\charup{12}{b} $, we have
\begin{align}\label{eqn:omegatilde}
\tilde\omega_D(Q_u\circ \hat{W}_d) 
&= \begin{dcases}
\charup{12}{r}\charup{D}{d}\omega_D(Q_u)&d=1,6,\\
-\charup{12}{r}\charup{D}{d}\omega_D(Q_u)&d=2,3,
\end{dcases}\nonumber\\
&=\mu(d)\charup{12}{r}\charup{D}{d}\charup{D}{u}.
\end{align}

\begin{proposition}\label{prop: central-cycle-reln} Let $ (r,6)=1 $ with $ 1\leq r\leq 23 $,  $ \lambda\geq 2 $ be an integer, and suppose $ {f\in  S_{2\lambda}(\Gamma_0(6))} $ with Atkin-Lehner signs $ \delta_p(r) $ for $ p\in\{2,3\} $. Let $ D $ be a fundamental discriminant with $ (-1)^\lambda D=\abs{D}\equiv r\pmod{24} $.   Then
\begin{align}\label{eqn: s cycle eq central val}
s_{\lambda,6}(f,D,(-1)^{\lambda}\abs{D}) =- 4i^{-\lambda}\varepsilon_r\charup{12}{r}\abs{D}^{\lambda-1/2}(2\pi)^{-\lambda}\Gamma(\lambda)
L \left(f\otimes\charup{D}{\cdot},\lambda\right).
\end{align} 
   
\end{proposition}

 \begin{proof}
 We first rewrite $ s_{\lambda,6} $ using \eqref{eqn: invariant CQ under involution}, \eqref{eqn: complete-class}, and \eqref{eqn:omegatilde} to obtain
 \begin{align*}
 s_{\lambda,6}(f,D,(-1)^\lambda \abs D) &= \sum_{d\mid 6}\sum_{u\pmod{D}}\tilde\omega_D(Q_u\circ \hat{W}_d)\int_{C_{Q_u\circ \hat{W}_d}}f(z)d_{Q_u\circ \hat{W}_d,\lambda}z\\
&=\charup{12}{r}\sum_{d\mid 6}\mu(d)\charup{D}{d}\delta_d(r)\sum_{u\pmod{D}}\charup{D}{u}\int_{C_{Q_u}}f(z)d_{Q_u,\lambda}z.
 \end{align*}
We have $ \delta_1(r)=1 $ and $ \delta_6(r) =\delta_2(r)\delta_3(r) $, so  
 \begin{gather*}
     \sum_{d\mid 6}\mu(d)\charup{D}{d}\delta_d(r) = \prod_{p\mid 6}\left(1-\charup{D}{p}\delta_p(r)\right)=4.
\end{gather*}
Recall that the orientation of $ S_{Q_u} $ depends on $ \operatorname{sign}(D)=(-1)^{\lambda} $; therefore
\begin{align*}
\sum_{u\pmod{D}}\charup{D}{u}\int_{C_{Q_u}}f(z)d_{Q_u,\lambda}z 
&= 
(-1)^{\lambda+1} \sum_{u\pmod{D}}\charup{D}{u}\int_{-u/D}^{i\infty} f(z)(Dz+u)^{\lambda-1}dz\\
&=(-1)^{\lambda+1} i\sum_{u\pmod{D}}\charup{D}{u}\int_0^\infty f \left(\frac{-u}
{D}+it\right)(Dit)^{\lambda-1}dt.
\end{align*} 

Write $ f(z)=\sum_{n\geq 1}b(n)e(nz) $. We have
\begin{align*}
\sum_{u\pmod{D}}\charup{D}{u}\int_{C_{Q_u}}f(z)d_{Q_u,\lambda}z
 &=
  i(Di)^{\lambda-1}(-1)^{\lambda+1} \sum_{u\pmod{D}}\charup{D}{u}\int_0^\infty f \left(\frac{u}{\abs{D}}+it \right)t^{\lambda-1}dt\\
&=-(-i)^\lambda D^{\lambda-1} \sum_{u\pmod{D}}\charup{D}{u}\int_0^\infty \sum_{n\geq 1}b(n)e \left(\frac{nu}{\abs{D}} \right)e^{-2\pi  n t}t^{\lambda-1}dt
\\
&=-(-i)^\lambda D^{\lambda-1} \int_0^\infty \sum_{n\geq 1}b(n)G\left(n,\charup{D}{\cdot}\right)e^{-2\pi  n t}t^{\lambda-1}dt,
\end{align*}
where $ G(n,\charup{D}{\cdot}) = \charup{D}{n}\sqrt{D} $ is the standard Gauss sum \cite[Prop. 3.4.10]{Cohen-Strom-MF-2017}. Hence
\begin{align*}
\sum_{u\pmod{D}}\charup{D}{u}\int_{C_{Q_u}}f(z)d_{Q_u,\lambda}z
&= 
-(-i)^\lambda D^{\lambda-1} \sqrt{D}\int_0^\infty \sum_{n\geq 1}\charup{D}{n}b(n)e^{-2\pi  n t}t^{\lambda-1}dt\\
&=-(-i)^{\lambda}\operatorname{sign}(D)^{1/2}D^{\lambda-1}\sqrt{\abs{D}}\int_0^\infty \left(f\otimes \charup{D}{\cdot}(it)\right)t^{\lambda-1}dt.
\end{align*}

The last integral is the Mellin transform of $ f\otimes \charup{D}{\cdot} $ evaluated at $ s=\lambda $, see \eqref{eqn:mellin-f}.
Now $ (-1)^\lambda = \operatorname{sign}(D) $ and $ \sqrt{(-1)^\lambda}D^{\lambda-1}\sqrt{\abs{D}} = \varepsilon_r \abs{D}^{\lambda-1/2} $ for all $ \lambda,r $ satisfying  \eqref{eqn:consistency}. Hence
\begin{align*}
s_{\lambda,6}(f,D,(-1)^\lambda \abs D)
 &=
 -4(-i)^\lambda\varepsilon_r\charup{12}{r}\abs{D}^{\lambda-1/2}(2\pi)^{-\lambda}\Gamma(\lambda)L \left(f\otimes\charup{D}{\cdot},\lambda\right).
\end{align*}

 \end{proof}

 We can now prove our main theorem.

 \begin{proof}[Proof of \cref{thm:Main-KZ-formula}]
    
Substituting \eqref{eqn: s cycle eq central val} into \eqref{eqn: coeff and cycle sum} when $ m=\abs{D} $ yields

\begin{align*}
\frac{\abs{a(\abs{D})}^2}{\langle F,F\rangle}
=
\frac{(-1)^\lambda \varepsilon_r^2 \abs{D}^{\lambda-1/2}(2\pi)^{-\lambda}\Gamma(\lambda)}{\sqrt{6}(12)^\lambda}\frac{L \left(f\otimes \charup{D}{\cdot},\lambda\right)}{\langle f,f\rangle},  
\end{align*}
and simplifying using $ \varepsilon_r^2=(-1)^{\lambda} $ through \eqref{eqn:consistency} completes the proof of \cref{thm:Main-KZ-formula}. 
 \end{proof}


\appendix
\section{Numerical Methods}
\subsection{Computing Central Values of \texorpdfstring{$L$}{L}-Functions}\label{sec:appx-numerical-central}

In this subsection we discuss the numerical approximation of central values of $ L $-functions of even weight eigenforms. The techniques are standard.\par

Let $ f\in S_{2\lambda}(\Gamma_0(N)) $ with Fricke eigenvalue $ \delta $. From \cite[Thm. 11.1.1]{Cohen-Strom-MF-2017} we have the completed $ L $-function 
\begin{align}\label{eqn:mellin-f}
\Lambda(f,s)&:=(2\pi)^{-s}\Gamma(s)L(f,s)
=\int_0^\infty t^{s-1}f(it)dt,
\end{align}
which is the Mellin transform of $ f $. 
Splitting the integral into the pieces $ \int_0^{1/\sqrt{N}}+\int_{1/\sqrt{N}}^\infty $ and applying the Fricke involution on the first term yields the representation
 \begin{align}
 L(f,s) = \frac{(2\pi)^s}{\Gamma(s)}\sum_{n\geq 1}b(n)\int_{1/\sqrt{N}}^{\infty}e^{-2\pi n t}(t^{s-1}+(-1)^\lambda \delta N^{\lambda-s}t^{2\lambda-s-1})dt.
 \end{align}
 This specializes to
 \begin{align}
 L(f,\lambda) = (1+(-1)^\lambda \delta)\frac{(2\pi)^\lambda}{\Gamma(\lambda)}\sum_{n\geq 1}b(n)\int_{1/\sqrt{N}}^{\infty}e^{-2\pi n t}t^{\lambda-1}dt.
 \end{align}

 One advantage of this integral representation is that it can approximate central values of twists relatively quickly using finitely many coefficients, once one knows the coefficients $ b(n) $. Some packages first need to compute the twist in terms of modular symbols, which can be computationally expensive because $ f\otimes \charup{D}{\cdot} \in S_{2\lambda}(\Gamma_0(ND^2))$.\par

\subsection{Computing Petersson Inner Products for Powers of the Eta Multiplier}\label{sec:appx-numerical-petersson}
Our numerical demonstration of \cref{thm:Main-KZ-formula} requires a method for approximating the Petersson inner product of elements in $S_{\lambda+1/2}(\slz,\nu_\eta^r)$ for $\lambda\in\mathbb{Z}_{\geq2}$. Numerous methods and algorithms exist for integral weight forms. We will adapt the method of \cite[Section 11.12]{Cohen-Strom-MF-2017}, which gives a beautiful and simple formula for the inner product in terms of the residue of a  Dirichlet series. \par 

Define the nonholomorphic Eisenstein series
\begin{align}
E_1(s)(\tau) : = \sum_{\substack{\gamma\in {\Gamma_\infty}\backslash{\Gamma_0(1)}\\\gamma=\abcd}}\frac{y^s}{\abs{c\tau+d}^{2s}}=\sum_{\gamma\in {\Gamma_\infty}\backslash \Gamma_0(1)}\Im(\gamma\tau)^s.
\end{align}
This series has a meromorphic continuation to $ \C $ with a simple pole at $ s=1 $, with residue $ \frac{3}{\pi} $. 
Then the methods of \cite[\S 11.12]{Cohen-Strom-MF-2017}, adapted to $ S_{\lambda+1/2}(\slz,\nu_\eta^r) $,  allow one to derive the following results.

\begin{proposition}\label{prop: convolv-L-fcn}
Let $F=\sum a(m)q^{m/24}$ and $G=\sum b(m)q^{m/24}$ be in $S_{\lambda+1/2}(\slz,\nu_\eta^r)$ and define the Dirichlet series
\begin{gather*}
S_2(F,G)(s):= \sum_{m\geq 1}\frac{a(m)\overline{b(m)}}{m^s}.
\end{gather*}
Then this series converges absolutely for $ \Re(s)>\lambda+1 $ and
\begin{gather*}
\int_{{\Gamma_\infty}\backslash{\mathbb H}}F(\tau)\overline{G(\tau)}y^{s+1}d\mu = \Gamma(s)\left(\frac{6}{\pi} \right)^{s}S_2(F,G)(s).
\end{gather*}
Moreover, $S_2 $ extends meromorphically via
\begin{align}\label{S2 repnE1}
\langle E_1(s)(\cdot)F(\cdot),G(\cdot)\rangle = \Gamma(\lambda-1/2+s)\left(\frac{6}{\pi} \right)^{\lambda-1/2+s}S_2(F,G)(\lambda+1/2+(s-1)),
\end{align}
and has at most a simple pole at $ s=\lambda+1/2 $ with residue  
\begin{align*}
\operatorname{Res}_{s=\lambda+1/2}(S_2(F,G)(s)) = \frac{3}{\pi}\frac{\langle F,G\rangle}{\Gamma(\lambda+1/2)}\left(\frac{\pi}{6} \right)^{\lambda+1/2}  .
\end{align*}

\end{proposition}

Therefore, we can compute inner products if we can compute the residue of $ S_2(F,G)(s) $. One heuristic for doing so gives 

\begin{align}\label{eqn: res S2}
\operatorname{Res}_{s=\lambda+1/2}(S_2(F,G)(s)) \approx \frac{\lambda+1/2}{X^{\lambda+1/2}} \sum_{m\leq X}a(m)\overline{b(m)}.
\end{align}

We briefly justify this heuristic. See also \cite{lowryduda2024numerical}. Let $ \kappa>\lambda+1 $. Then, for $ X\not\in\Z $, the effective Perron formula \cite[Part II, Thm. 2.3]{tenen-intro} gives 
\begin{align*}
\sum_{m\leq X}a(m)\overline{b(m)} = \frac{1}{2\pi i}\int_{\kappa-i T}^{\kappa+iT} S_2(F,G)(s)\frac{X^s}{s}ds + R(X,T),
\end{align*} 
where $ R(X,T)=O \left(X^\kappa \sum_{m\geq 1}\frac{\abs{a(m)\overline{b(m)}}}{m^\kappa (1+T\abs{\log(X/m)})}\right) $.  Shifting the line of integration to the left and picking up the residue of $ S_2(F,G)(s) $ at $ \lambda+1/2 $   shows that
\begin{align*}
\sum_{m\leq X}a(m)\overline{b(m)} \approx \operatorname{Res}_{s=\lambda+1/2}(S_2(F,G)(s))\frac{X^{\lambda+1/2}}{{\lambda+1/2}},
\end{align*}
assuming bounds on $ R(X,T) $ and the line integrals resulting from shifting the line of integration. Bounding these sources of error would make the following heuristic precise.

\begin{heuristic}\label{prop: inner product numerical}
   
Assume the conditions and notation of \cref{prop: convolv-L-fcn}.  We have, for suitably large $ X>0 $, the approximation
\begin{align}\label{eqn:perron-eff}
\langle F,G\rangle \approx \frac{\pi}{3}\Gamma(\lambda+1/2)\left(\frac{6}{\pi} \right)^{\lambda+1/2}\frac{\lambda+1/2}{X^{\lambda+1/2}}\sum_{n\leq X}a(n)\overline{b(n)}. 
\end{align}

\end{heuristic}

\bibliographystyle{alpha}
\bibliography{ref.bib}
\end{document}